\documentclass[11pt, oneside]{amsart}
\usepackage{amsmath}
\usepackage{amssymb}
\usepackage{color}
\usepackage[usenames,dvipsnames,x11names,svgnames]{xcolor}
\usepackage[colorlinks=true,linkcolor=NavyBlue,citecolor=DarkGreen, urlcolor=blue]{hyperref}
\usepackage{amsthm}
\usepackage{amscd}
\usepackage{geometry}
\usepackage{mathrsfs}
\usepackage{dsfont}
\usepackage{mathtools}

\numberwithin{equation}{section}
\mathtoolsset{showonlyrefs=true}
\newtheorem{thm}{Theorem}
\newtheorem*{ethm}{Theorem}

\newtheorem{prop}{Proposition}
\newtheorem{lem}[prop]{Lemma}

\theoremstyle{definition}

\newtheorem*{rem}{Remark}

\newtheorem*{ack}{Acknowledgments}

\newcommand{\mb}{\mathbb}
\newcommand{\mc}{\mathcal}
\newcommand{\mf}{\mathfrak}
\newcommand{\ol}{\overline}
\newcommand{\ul}{\underline}
\newcommand{\leqs}{\leqslant }
\newcommand{\geqs}{\geqslant }

\newcommand{\be}{\begin{equation*}}
\newcommand{\ee}{\end{equation*} }

\newcommand{\ben}{\begin{equation}}
\newcommand{\een}{\end{equation} }

\newcommand{\bs}{\begin{split}}
\newcommand{\es}{\end{split}}

\newcommand{\bmu}{\begin{multline*}}
\newcommand{\emu}{\end{multline*}}

\newcommand{\bmun}{\begin{multline}}
\newcommand{\emun}{\end{multline}}

\begin{document}

\title{Simultaneous extreme values of zeta and $L$-functions}
\author[W. Heap]{Winston Heap}
\address{Department of Mathematical Sciences, Norwegian University of Science and Technology (NTNU), 7491 Trondheim, Norway.} 
\email{winstonheap@gmail.com}
\author[J. Li]{Junxian Li}
\address{
Universit\"at Bonn,
Mathematisches Institut,
Endenicher Allee 60,
53115 Bonn,
Germany.}
\email{jli135@math.uni-bonn.de}
\maketitle 

\begin{abstract}
We show that distinct primitive $L$-functions can achieve  extreme values \emph{simultaneously} on the critical line. 
Our proof uses a modification of the resonance method and can be applied to establish simultaneous extreme central values of $L$-functions in families.
\end{abstract}

\section{Introduction}
Extreme values of $L$-functions have attracted considerable attention in recent years.  
This uptake in activity is largely due to the introduction of the resonance method of Soundararajan \cite{Sound res} and its subsequent developments. This is a versatile method which allows one to show that for various families of $L$-functions $\mathcal F$ with analytic conductor $C$, 
\begin{align}\label{large}
\max_{\substack{\pi \in \mathcal F\\ \operatorname{cond}\pi \asymp C}}|L(\tfrac12,\pi)|\geqs \exp\Big(c\sqrt{\frac{\log C}{\log\log C}}\Big)
\end{align}
for some constant $c=c(\mathcal F)>0$. In the case of the Riemann zeta function, Bondarenko--Seip \cite{BS} recently made the significant improvement 
\begin{align}\label{larger}
\max_{t\in[0, T]}|\zeta(\tfrac12+it)|\geqs \exp\Big( c '\sqrt{ \frac{\log T\log\log \log T}{\log\log T}}\Big).
\end{align}
Their modification of the resonance method can be extended to other families \cite{dlBT}, although a severe restriction is that the $L$-functions must have non-negative coefficients\footnote{This requirement can be weakened slightly; the method applies as long as the resultant mean values contain non-negative terms. See the proof of  Theorem 1.5 of \cite{dlBT} for example.} . 

%

Extreme values of products of $L$-functions are also possible and in particular the bound \eqref{large} holds when $L(s,\pi)$ factorises. For example, we know \cite[Theorem 1.12]{BFKMMS} that for fixed cusp forms $f, g$ there exists a non-trivial primitive character $\chi$ modulo a prime $q$ such that  
\begin{align}\label{fgchi}
|L(1/2, f \otimes \chi)L(1/2, g\otimes \chi)|\geqs \exp\Big (c_{f, g}\sqrt{\frac{\log q}{\log \log q}}\Big).
\end{align} 
In $t$-aspect, Aistleitner--Pa\'nkowski \cite{AP} showed that \eqref{large} holds for non-primitive $L$-functions in the Selberg class, whereas bounds of the strength \eqref{larger} can be demonstrated for Dedekind zeta functions \cite{BDHHS}. These results contain an interesting feature: when $L(s,\pi)$ is a product of $L$-functions one can achieve a larger constant $c$. Precisely, the $t$-aspect results show that for a product of $m$ distinct primitive $L$-functions $c=\sqrt{m}$ is admissible, whereas for individual $L$-functions $c=1$.  


As pointed out in \cite{BDHHS}, this presents a dichotomy. Either extreme values of $L$-functions are occurring \emph{simultaneously}, or one can demonstrate even larger values for an individual $L$-function. The question of whether $L$-functions can attain extreme values simultaneously is natural and forms the focus of this paper. 
Another motivation for the present work is that extreme values provide a testing ground for the statistical independence of $L$-functions; a property which, if held at the extremes, can have deep arithmetic consequences\footnote{For example, from the work of Iwaniec--Sarnak \cite{IS} we know that if the small values of $L(1/2,f)$ and $L(1/2,f\otimes\chi_d)$ do not know about each other (in a specified sense) then one can rule out the existence of Landau--Siegel zeros.}.


Previous results on joint extreme values have been obtained when we move away from the central point. For fixed $1/2<\sigma<1$, under suitable assumptions on the $L$-functions Mahatab--Pa\'nkowski--Vatwani \cite{MPV} used Diophantine approximations to show the existence of large values of $L_j(\sigma+it_j)$ with the $t_j$'s living in a small neighbourhood of each other. By considering the joint value distributions of $L(\sigma+it, \chi_j)$, Inoue--Li \cite{IL} recently showed the existence of \emph{simultaneous} large values of Dirichlet $L$-functions in the strip $1/2< \sigma<1$. 

Less is known for joint extreme values at the central point. 
It was observed by Selberg \cite{Sel old} and later proved by Bombieri--Hejhal \cite{BH} that the joint distribution of $L$-functions (satisfying suitable assumptions) splits as the product of the distributions: 
\begin{align}\label{BH}
\frac{1}{T}\mathrm{meas}\bigg\{t\in[T,2T]:\frac{\log|L_j(\tfrac12+it)|}{\sqrt{\tfrac12\log\log T}}\geqs V_j,\,\, j=1,\ldots, m\bigg\}
\sim
\prod_{j=1}^m \int_{V_j}^\infty e^{-x^2/2}\frac{dx}{\sqrt{2\pi}}
\end{align}
as $T\to\infty$ for any fixed $V_j$.  
Recently Inoue--Li \cite{IL0} extended the range of $V_j$ in \eqref{BH} to $V_j\ll (\log\log T)^{1/10}$ which yields simultaneous large values of size $\exp \big( c(\log \log T)^{3/5-\epsilon}\big)$. Simultaneous extreme values of size \eqref{large} would be desirable, but they are harder to detect using joint distributions since these events are very rare. Indeed, even bounds for the joint distribution in the range $V_j\approx \sqrt{\log\log T}$ are not available unconditionally due to our lack of knowledge on mixed moments
(although the conjectured asymptotics \cite{Hdedekind, MTB} could likely be established up to order conditionally \cite{Ha, HS}). Here, we utilise a modification of the resonance method to detect simultaneous large values of central values of $L$-functions, overcoming inefficiencies from methods of moments and Diophantine approximation. Our first result gives simultaneous large values of two $L$-functions on the critical line.

\begin{thm}\label{main thm}
Let $L_1(s), L_2(s)$ be either the Riemann zeta function, a primitive Dirichlet $L$-function or the $L$-function attached to a primitive cusp (holomorphic or Maa\ss ) form on $GL(2)$ over $\mathbb{Q}$. Then there exists some positive constant $c$ depending on $L_1, L_2$ such that for sufficiently large $T$,  we have
\[
\max_{t\in[T,2T]}\min\big(|L_1(\tfrac12+it)|,|L_2(\tfrac12+it)|\big)\geqs \exp\bigg(c\sqrt{\frac{\log T}{\log\log T}}\bigg).
\]
\end{thm}
\begin{rem}
	The constant depends on the degree of the $L$-functions in question. If $L_1, L_2$ are both Dirichlet $L$-functions, then we can take $c=\sqrt{17/66}+o(1)$ and if at least one of $L_1$ and  $L_2$ is a $GL(2)$ $L$-function, we can take $c=\sqrt{ {(1-2\theta)}/{12}}+o(1)$ where $\theta$ is an admissible bound towards the Ramanujan conjecture for $GL(2)$ over $\mathbb{Q}$. By work of Kim--Sarnak \cite[Appendix 2]{KS}, we can take $\theta=7/64$. For comparison, we mention that one can take $c=\sqrt{2}+o(1)$ for the product of $L_1L_2(1/2+it)$  (see \cite{AP}). 
\end{rem}

Our method extends to other families and we shall describe a general principle below. For now, we illustrate this by demonstrating simultaneous large central values of twists of $GL(2)$ cusp forms, refining \eqref{fgchi}.  
\begin{thm}\label{mod q twist thm}
	Let $f, g$ be fixed primitive cusp forms of level $r, r'$ and trivial central character. There exists a positive constant $c$ depending only on $f, g$ such that for all primes $q$ sufficiently large in terms of $f, g$, there exists a non-trivial character $\chi \bmod q$ such that 
	\begin{align*}
	\min\Big(|L(1/2, f\otimes \chi)|, |L(1/2, g\otimes \chi)|\Big)\geqs \exp \Big(c \sqrt{\frac{\log q}{\log\log q}}\Big).
	\end{align*}
\end{thm}
\begin{rem}
	We give an explicit constant $c$ in terms of $f, g$ in the proof. In a generic situation, we can take $c={1}/{12\sqrt{10}}+o(1)$, which is half of the constant $c_{f, g}={1}/{6\sqrt{10}}+o(1)$ for the product in \eqref{fgchi}  \cite[Remarks 7.3, 7.20]{BFKMMS}, as one would expect. 
\end{rem}

We now describe our method. The idea is to use a resonator that picks out a large value of the product of the $L$-functions that, at the same time, is significantly bigger than their sum, thus giving simultaneous large values.
We detail this in the $t$-aspect, although the principle extends more generally. Our aim is to find a Dirichlet polynomial $R(t)$ such that for large $V$,
\begin{equation}\label{resineq}
\int_T^{2T}\Big(|L_1(\tfrac12+it)L_2(\tfrac12+it)|^2-V|L_1(\tfrac12+it)|^2-V|L_2(\tfrac12+it)|^2\Big)|R(t)|^2dt>0.
\end{equation}
If this holds then there exists a $t\in[T,2T]$ such that 
\begin{equation}\label{key ineq}
|L_1(\tfrac12+it)L_2(\tfrac12+it)|^2-V|L_1(\tfrac12+it)|^2-V|L_2(\tfrac12+it)|^2>0
\end{equation}
which implies that both $|L_1(1/2+it)|^2,|L_2(1/2+it)|^2>V$. 
We choose $R(t)$ to pick out large values of the product and once the asymptotic formulae for twisted second moments in \eqref{resineq} have been established, we can find the desired size for $V$. This approach uses the larger values of the product in a key way, but also includes the required upper bound information (as is necessary to rule out excessively large values of individual $L$-functions).

For multiple $L$-functions we can aim to find a value of $t$ for which
\begin{equation}\label{gen ineq}
\prod_{j=1}^m |L_j(\tfrac12+it)|^2-V\sum_{1\leqs i\leqs m}\prod_{\substack{j=1\\j\neq i}}^m|L_j(\tfrac12+it)|^2 >0.
\end{equation}
Unfortunately, asymptotic formulae for the twisted second moments of multiple or higher degree $L$-functions are currently out of reach.  
However, asymptotics are not strictly necessary since reasonably sharp bounds would suffice. A lower bound for the first term on the left of \eqref{gen ineq} can be achieved fairly easily through the Cauchy--Schwarz inequality. 
 This allows one to replace the second moment with a first moment which is more tractable. In fact, this first moment can be computed for any number of $L$-functions. 

To obtain upper bounds for the second term in \eqref{gen ineq} we note that there exist several instances in the literature \cite{DFL, GZ, LR} where, if one doesn't have access to asymptotics for twisted second moments, a sharp upper bound can still be achieved on the Riemann Hypothesis by applying the methods of Harper \cite{Ha}.  
As it stands, Harper's method is designed for a fixed $2k$-th moment of an $L$-function and thus sensitive to values of size $(\log T)^k$. 
However, with some modifications, in particular by focusing on very small primes, these methods can be made suitable for extreme values.  
With this in hand we can exhibit simultaneous large values for many $L$-functions in higher degrees under the Riemann Hypothesis for these $L$-functions. 

\begin{thm}\label{main thm 2} Let $\pi_j$, $j=1,\cdots, m$ be irreducible unitary cuspidal
	automorphic representations of $GL(d_j)$ over $\mathbb{Q}$ such that $\pi_i\not \cong \pi_j$ for $i\not=j$. Assume the generalised Riemann Hypothesis for all $L(s, \pi_j)$ $j=1, \dots, m$ and assume the generalised Ramanujan conjecture for $L(s, \pi_j)$ if $d_j\geqs 3$.
	 Then for sufficiently large $T$ we have
\[
\max_{t\in[T,2T]}\min(|L(\tfrac12+it,\pi_1)|,\ldots,|L(\tfrac12+it,\pi_m)|)\geqs \exp\bigg(c\sqrt{\frac{\log T}{\log\log T}}\bigg)
\]
for any positive constant $c< \frac{1}{\sqrt{2m}}$. 
\end{thm}


Additional work is required to remove the generalised Ramanujan conjecture in the case of Maa\ss\ forms. Throughout most of the proof we can work under weaker assumptions, in particular when computing the mean values. However in the final step when establishing extreme values more strict control on the size of $a_\pi(p)$ is required. One sufficient condition is a Mertens' type estimate for the fourth power moment: 
\begin{align*}
\sum_{p\leqs X}\frac{|a_\pi(p)|^4}{p}\ll \log\log X.
\end{align*}
Therefore, the assumption of the generalised Ramanujan conjecture can be avoided for $GL(2)$ Maa\ss\  forms by the functionality of symmetric powers established by Kim \cite{Kim} (and also for self-dual $GL(3)$ $L$-functions using Gelbart-Jacquet \cite{GJ} and Kim \cite{Kim}, although we do not state this in Theorem \ref{main thm 2} for concision). 
 

We also remark that Theorem 3 can be proved unconditionally for three distinct Dirichlet $L$-functions. Here, the upper bounds for the twisted second moments of $L(1/2+it,\chi_i)L(1/2+it,\chi_j)$ are (essentially) available being generalisations of the twisted fourth moment of the Riemann zeta function \cite{BBLR, Htwist,HY}.

The versatility of the resonance method in families of $L$-functions extends to simultaneous values, both large and small. Our methods allow for a general principle which we now describe. 
Let ${\Pi}$ be some family and let $\{L(s,f_j)\}_{j=1}^m$ be fixed $L$-functions. Suppose we can lower bound  
\begin{equation}\label{first moment}
\sum_{\pi\in\Pi} \prod_{j=1}^mL(1/2, f_j\otimes \pi)|R(\pi)|^2
\end{equation}
and upper bound  
\begin{equation}\label{abs second mmt}
\sum_{\pi\in\Pi} \prod_{j\neq i}|L(1/2, f_j\otimes \pi)|^2|R(\pi)|^2
\end{equation}
in a reasonably sharp way. Then via the inequality \eqref{gen ineq} one can show the existence of $\pi\in \Pi$ such that $|L(1/2,f_j\otimes \pi)|$ are large simultaneously. Here, one uses the Cauchy--Schwarz inequality to get a lower bound for the mixed absolute second moment using  \eqref{first moment} together with an estimate for $\sum_{\pi \in \mathcal F}|R(\pi)|^2$. If upper bounds for \eqref{abs second mmt} are not immediately accessible, then one can apply our adaption of Harper's methods (Section \ref{cond upper sec}) to give conditional results.

The situation for simultaneous small values is somewhat simpler. Here, we just require upper bounds for the sum 
\begin{equation}\label{abs second mmt 2}
\sum_{\pi\in\Pi} \sum_{j=1}^m|L(1/2, f_j\otimes \pi)|^2|R(\pi)|^2
\end{equation}
along with the simple fact that for non-negative $a,b$, the inequality  $a+b\leqs V$ implies $a,b\leqs V$. 
In both the large and small value cases, the resonator should be chosen to pick out extreme values of the full product of $L$-functions.  

We illustrate this principle in two families. As we have seen already in Theorem \ref{mod q twist thm}, it can be applied to give simultaneous large central values of $L$-functions of $GL(2)$ cusp forms twisted by Dirichlet characters modulo $q$ unconditionally, since the second moment theory has been well developed \cite{BFKMMS}.


For simultaneous small values, we consider the family of holomorphic cusp forms twisted by quadratic characters $\chi_d(\cdot)=(\frac{d}{\cdot})$. Here, the mixed first moment
\[
\sum_{\substack{0<d\leqs X}}L(1/2,f\otimes \chi_{d})L(1/2,g\otimes \chi_{d})
\] 
 is still unknown (though significant progress has been made recently by X. Li \cite{XL}). 
Consequently, the simultaneous non-vanishing of quadratic twists of cusp forms remains an open question. Nevertheless, we can show that there are infinitely many $d$ such that $L(1/2,f\otimes \chi_{d})$ and $L(1/2,g\otimes \chi_{d})$ get very small simultaneously. 

\begin{thm}\label{quad twist thm}Let $f,g$ be holomorphic cusp forms of weight $\kappa\equiv 0 \bmod 4$ for $SL_2(\mathbb{Z})$ and let $\chi_{d}(\cdot)=(\frac{d}{\cdot})$ be the Kronecker symbol. Then for large $X$ there exists $X\leqs d\leqs 2X$ and some $c>0$ such that 
	\[
	\max(L(1/2, f\otimes \chi_{d}),L(1/2,g\otimes \chi_{d}))\ll \exp\bigg(-c\sqrt{\frac{\log X}{\log\log X}}\bigg). 
	\]
\end{thm}
We remark that this result is unconditional since we can avoid the absolute second moments in \eqref{abs second mmt 2} (which are currently out of reach) and work directly with $L(1/2, f\otimes \chi_{d})$ as we already have non-negativity: $L(1/2, f\otimes \chi_{d})\geqs 0$. This surprising fact is known unconditionally from the formula of Waldspurger \cite{Wa} (see also \cite{KZ}).
In generic situations, one can take $c={1}/{\sqrt{5}}+o(1)$.

There are several other possibilities for simultaneous extreme values in families of $L$-functions. Examples of significant arithmetic interest are given by the families
\[
\{L(1/2,f), L(1/2,f\otimes \chi_D): f\in\mathcal{F} \}
\]
where $\chi_D$ is a fixed quadratic character and $\mc{F}$ is either the family of Hecke eigencuspforms of even weight $k$ for the full modular group with $k$ tending to infinity, or the family of holomorphic newforms of fixed even weight $k$ for the congruence subgroup $\Gamma_0(N)$ with $N$ tending to infinity. In both the weight and (squarefree) level aspects, the required second moment formulae can be computed using  Petersson's formula (see \cite{D, IS, KMV, Sound res} for example). 
We also mention that for a large prime $q$, Dirichlet characters $\omega_1, \omega_2$ modulo $q$ and $f$ a Hecke eigenform for $SL_2(\mb{Z})$ (holomorphic or Maa\ss\ \!\!), the simultaneous extreme values for the families 
\begin{align*}
\{L(1/2, \chi), L(1/2, \omega_1\chi), L(1/2, \omega_2\chi): \chi \bmod q\}\\
\{L(1/2, f\otimes \chi), L(1/2, \chi): \chi \bmod q\}
\end{align*}
could be established using work of Zacharias \cite{Z} under GRH. 

We close this introduction with a few remarks on similarities with previous works and on the difficulties in extending our results to the strength of \eqref{larger}.  
We note that our proof utilises some control on both upper and lower bounds for $L$-functions. A similar idea has appeared in the recent work of Gun--Kohnen--Soundararajan \cite{GKS} where they demonstrated large central values of linear combinations of $L$-functions by making one $L$-function large and at the same time keeping all other $L$-functions smaller. In contrast, we  exhibit simultaneous extreme central values in families of $L$-functions where all of the $L$-functions attain large or small values.    

A very natural question is whether one can attain simultaneous values of the strength of Bondarenko--Seip \cite{BS} given in \eqref{larger}.  A key source of their improvement was the use of a resonator with support in numbers that are much bigger than $T$, although this results in considerable difficulties. First, to lower bound 
\[
\int_1^{T}L_1(\tfrac12+it)L_2(\tfrac12+it)|R(t)|^2dt
\]
with such a resonator we can require that the coefficients of $L_1(s)L_2(s)$ be positive after inserting a smooth weight with positive transform,  as in \cite{BS}. This could be resolved with Dedekind zeta functions, for example.   
However, finding an upper bound for  
\[
\int_1^T|L_j(\tfrac12+it)|^2 |R(t)|^2dt 
\] 
when $R$ is a long resonator seems a more substantial obstacle. In \cite{BS}, to get an upper bound for $\int |R(t)|^2dt$, the resonator was chosen to have well-spaced phases but unfortunately it is not clear how such an $R$ interacts with $|L_j(1/2+it)|^2$. If this issue could be overcome then one could deduce bounds of the form \eqref{larger} for Dirichlet $L$-functions, or indeed any other $L$-function with negative coefficients provided they appear as a factor in a Dedekind zeta function. 

\begin{ack}We would like to thank Edgar Assing and Peter Humphries for helpful discussions. We also thank Jesse Thorner for valuable remarks on a preliminary version of this paper.  

\end{ack}


\section{Background on automorphic $L$-functions}

In this section we collect some basic facts about the class of $L$-functions used in our $t$-aspect results. These can be found in many places, see for example \cite{MTB, RS}.  

Let $L(s,\pi)$ be the $L$-function attached to an irreducible cuspidal automorphic representation $\pi$ of $\mathrm{GL}(d)$ over $\mb{Q}$ normalised such that $\pi$ has unitary central character.
In the region $\sigma>1$ we have 
\begin{equation}\label{L series/prod}
L(s,\pi)
=
\sum_{n=1}^\infty \frac{A_\pi(n)}{n^s}
=
\prod_p\prod_{j=1}^d\bigg(1-\frac{\alpha_{\pi, j}(p)}{p^s}\bigg)^{-1}
\end{equation}
for some 
complex coefficients $A_\pi(n)$ and $\alpha_j(p)$. If $d=1$ and $\pi$ is the trivial representation then $L(s,\pi)$ is given by the Riemann zeta function. Otherwise, it extends to an entire function satisfying the functional equation 
\[
\Phi(s,\pi):=N^{s/2}\gamma(s,\pi)L(s,\pi)=\epsilon_\pi \overline{\Phi}(1-s,\pi) 
\]
where $N\in\mb{N}$, $|\epsilon_\pi|=1$, $\overline{\Phi}(s,\pi)=\overline{\Phi(\ol{s},\pi)}$ and 
$
\gamma(s,\pi)=\pi^{-ds/2}\prod_{j=1}^d \Gamma\Big(\frac{s+\mu_{\pi, j}}{2}\Big)
$
for some complex numbers $\mu_{\pi, j}$ satisfying $\Re\mu_{\pi, j}>-1$. 


Applying Stirling's formula to $\gamma(s,\pi)$ along with the Phragmen--Lindel\"of principle, we see that 
\begin{equation}\label{phrag}
L(\sigma+it,\pi)\ll (N|t|^d)^{(1-\sigma)/2+\epsilon}
\end{equation}
in the strip $-\delta\leqs \sigma\leqs 1+\delta$ for large $|t|$ (see \cite{IK} for example). 

In our conditional results we make key use of Euler products, so we collect some further practical bounds here. For $\sigma>1$ on differentiating the Euler product we see
\begin{align}\label{apidef}
-\frac{L^\prime}{L}(s,\pi) 
=
\sum_{p^\ell\geqs 2}\frac{\log p\sum_{j=1}^d \alpha_{\pi, j}(p)^\ell}{p^{\ell s}}
=:
\sum_{n\geqs 2}\frac{\Lambda(n) a_\pi(n)}{n^s}
\end{align} where $\Lambda(n)$ is the von-Mangoldt function.

The generalised Ramanujan conjecture asserts that $|\alpha_{\pi, j}(p)|=1$ for all but a finite number of primes and satisfies $|\alpha_{\pi,j}(p)|\leqs 1$ elsewhere, although this remains open in general. Rudnick--Sarnak \cite{RS} have shown that
\begin{equation}\label{RS bound}
|\alpha_{\pi, j}(p)|\leqs p^{1/2-1/(d^2+1)}
\end{equation}
for all primes $p$. 
This bound implies that 
\begin{equation}\label{a bound}
|a_\pi(p^\ell)|=|\sum_{j=1}^d \alpha_{\pi, j}(p)^\ell|\leqs d p^{\ell(1/2-1/(d^2+1))}
\end{equation}
and, by \eqref{L series/prod}, that
\begin{equation}\label{coeff bound}
|A_\pi(n)| \leqs \tau_d(n)n^{1/2-1/(d^2+1)}
\end{equation}
where $\tau_d$ is the generalised divisor function. Another useful bound is that when the degree satsifies $d\leqs 3$, 
\begin{equation}\label{a pp bound}
\qquad |a_j(p^\ell)|\ll 1+ |a_j(p)|^\ell,\qquad d\leqs 3
\end{equation}
for fixed prime powers $\ell$. This is given in the proof of Proposition 2.4 of \cite{RS}.

In many situations, the assumption of the generalised Ramanujan conjecture can be replaced by Hypothesis H introduced by Rudnick-Sarnak \cite{RS}, which states that for fixed $\ell\geqs 2$, 
\begin{equation}\label{H hyp} 
\sum_{p}\frac{| a_\pi(p^\ell)|^2(\log p)^2}{p^\ell}<\infty.
\end{equation}
Clearly, this follows from the the generalised Ramanujan conjecture. Unconditionally, this was shown to hold for $d\leqs 3$ by Rudnick--Sarnak \cite{RS} using \eqref{a pp bound} and for $d=4$ by Kim \cite{Kim}. 
As a replacement for the generalised Ramanujan conjecture, Hypothesis H has previously been used in mean value results \cite{MTB}. In our case it could be used to weaken the assumptions of Proposition \ref{gen upper bound prop} below, but at the cost of considerable extra technicalities. We settle for using \eqref{a pp bound}, which restricts our unconditional results to $d\leqs 3$ rather than $d\leqs 4$, since in the end we require stronger coefficient bounds to deduce our final results.          

The following three results are the key average properties we require for $a_\pi(p)$, all of which hold unconditionally when $d\leqs 2$.
The first is the orthogonality conjecture of Selberg, as shown in full generality by Liu--Ye \cite{LY} building on previous works \cite{LY0, LY1, LiuYe}.

\begin{ethm}[Selberg's orthogonality conjectures] Let $\pi$, $\pi^\prime$ be two irreducible unitary cuspidal automorphic representations of $\mathrm{GL}(d)$, $\mathrm{GL}(d^\prime)$ over $\mb{Q}$, respectively. If \eqref{H hyp} holds, then
 \begin{equation}\label{Selberg conj 1}
 \sum_{p\leqs x}\frac{a_\pi(p)\ol{a_{\pi^\prime}(p)}}{p}
 =
\begin{cases}
\log\log x+O(1) & \text{if} \,\,\,\,\,\pi \cong \pi^\prime,
\\
O(1)  & \text{otherwise.} 
\end{cases}
 \end{equation}
 In particular, \eqref{Selberg conj 1} holds if $\max(d,d')\leqs 4$ or on assuming the generalised Ramanujan conjecture.
\end{ethm}

This essentially implies the following bounds for our resonator sums. 

\begin{prop}\label{res sum bound cor} Let $\pi$, $\pi^\prime$ be two irreducible unitary cuspidal automorphic representations of $\mathrm{GL}(d)$, $\mathrm{GL}(d^\prime)$ over $\mb{Q}$, respectively. If \eqref{H hyp} holds then for large $x,y$,
 \begin{equation}\label{res sum bound}
 \sum_{x<p\leqs y}\frac{a_\pi(p)\ol{a_{\pi^\prime}(p)}}{p\log p}
 =
\begin{cases}
 \frac{1}{\log x}- \frac{1}{\log y}+O\Big(\frac{1}{(\log x)^2}\Big) & \text{if} \,\,\,\,\,\pi \cong \pi^\prime,
\\
O\Big(\frac{1}{(\log x)^2}\Big) & \text{otherwise.} 
\end{cases}
 \end{equation}
 In particular, this holds if $\max(d,d^\prime)\leqs 4$ or on assuming the generalised Ramanujan conjecture. 
\end{prop}
\begin{proof}
The sum is given by 
\[
 \sum_{x<p\leqs y}\frac{a_\pi(p)\ol{a_{\pi^\prime}(p)}}{p\log p}
 =
 \sum_{x<n\leqs y}
 \frac{\Lambda(n)a_\pi(n)\ol{a_{\pi^\prime}(n)}}{n(\log n)^2}
 -
 \sum_{x<p^\ell\leqs y,\ell\geqs 2}\frac{a_\pi(p^\ell)\ol{a_{\pi^\prime}(p^\ell)}}{\ell^2p^\ell \log p}.
\] 
 To estimate the second sum we note that the contribution from $\ell\geqs d^2+1$ can be bounded using \eqref{coeff bound} by
 \[
 \ll \frac{1}{(\log x)^2}\sum_{p}\sum_{\ell\geqs d^2+1}\frac{\log p}{p^{2\ell/d^2+1}}\ll \frac{1}{(\log x)^2}.
 \] 
For $\ell\leqs d^2+1$ we use Cauchy-Schwarz and \eqref{H hyp} to obtain
  \[
 \ll \frac{1}{(\log x)^2}\sum_{\substack{x<p^\ell\leqs y\\2\leqs \ell\leqs d^2+1}}\frac{\log p|a_\pi(p^\ell)\ol{a_{\pi^\prime}(p^\ell)|}}{p^{\ell}}\ll \frac{1}{(\log x)^2}.
 \] 
 To compute the first sum we apply partial summation along with the bounds 
\begin{equation}\label{LY pnt}
S(x):=\sum_{n\leqs x}\frac{(\log n)\Lambda(n)a_\pi(n)\ol{a_{\pi^\prime}(n)}}{n}=
\begin{cases}
\tfrac12(\log x)^2+O(\log x) & \text{if} \,\,\,\,\,\pi \cong \pi^\prime,
\\
O(\log x) & \text{otherwise.} 
\end{cases}
\end{equation}
of  \cite{LY}.
\end{proof}

We shall use one more type of bound, which can be used to avoid the assumption of the generalised Ramanujan conjecture in some cases. 

\begin{ethm}[Fourth moment bounds] Suppose $d\leqs 2$ or that $d=3$ and $\pi$ is self-dual. Then
\begin{align}\label{fourth moment}
\sum_{p\leqs x}\frac{|a_\pi(p)|^4}{p}\ll \log\log x.
\end{align}
\end{ethm}
\begin{proof}
For $d=1$ the result is clear. For $d=2$ this follows from the fact that 
\[
a_\pi(p)^4=2+a_{\mathrm{Sym}^2 \pi}(p)+a_{\mathrm{Sym}^4 \pi}(p),
\]
(see the proof of Corollary 2.15 of \cite{BFKMMS} for example) along with the bounds $$\sum_{p\leqs x}a_{\mathrm{Sym}^k \pi}(p)/p\ll \log \log x, \ k=2, 4.$$
When $d=3$ and $\pi$ is self-dual it is known (see Section 3.2 of \cite{JL}) that $L(s,\pi\times\pi\times\pi\times\pi)$ has a pole of order 3 at $s=1$. The result in this case therefore follows by Tauberian theorems. 
\end{proof}


\section{Simultaneous large values in $t$-aspect:\\ set-up and proofs of Theorems \ref{main thm} and \ref{main thm 2}}

In this section we give the set-up for proving simultaneous large values in the $t$-aspect, state the required moment bounds and then complete the proofs of Theorems \ref{main thm} and \ref{main thm 2}. 
Let $\pi_i$, $1\leqs i\leqs m$ be irreducible unitary cuspidal automorphic representations of $\mathrm{GL}(d_i)$ over $\mb{Q}$ such that $\pi_i\not \cong \pi_j$ for $1\leq i\not=j\leq m$, respectively, and let
\[
L_i(s)=L(s,\pi_i)=\sum_{n=1}^\infty \frac{A_{\pi_i}(n)}{n^s}
\] 
be the associated $L$-functions. For brevity we denote $a_i(p)=a_{\pi_i}(p)=A_{\pi_i}(p)$.

To pick out simultaneous values we recall that if there exists a $t$ such that 
\begin{equation}
\prod_{i=1}^m |L_i(1/2+it)|^2-V\sum_{1\leqs i\leqs m}\prod_{\substack{j=1\\j\neq i}}^m|L_j(1/2+it)|^2 >0,
\end{equation}
 then we must have $|L_i(1/2+it)|^2>V$ for all $1\leqs i\leqs m$. Write
 \[
 L(s)
 =
 \prod_{i=1}^m L_i(s)
 =
 \sum_{n\geqs 1}a(n)n^{-s}
 \]
 so that 
 \[
 a(p)=\sum_{i=1}^m A_{\pi_i}(p)=\sum_{i=1}^ma_{i}(p).
 \]
 
We choose our resonator to pick out large values of $L(s)$. Let $X=T^\Delta$ for some $\Delta<1$ to be chosen and denote
\[
\mc{L}=\sqrt{\frac{1}{m}\log X\log\log X}.
\]
For small $\epsilon>0$ let
\[
\mc{P}=\bigg\{\mc{L}^2<p\leqs \exp((\log \mc{L})^2): |a_i(p)|\leqs (\log p)^{1-\epsilon} \text{ for all } 1\leqs i\leqs m\bigg\}.
\]
We then define $r(n)$ to be the multiplicative function supported on squarefree numbers for which  
\begin{equation*}\label{r}
r(p)
=
\begin{cases}
a(p)\frac{\mc{L}}{\sqrt{p}\log p}{}, \,\,\, &\text{ for } p\in\mc{P,} 
\\
0,  & \text{ otherwise. }
\end{cases}
\end{equation*}
and let
\begin{align}\label{Rtdef}
R(t)=\sum_{n\leqs X}r(n)n^{-it}.
\end{align}
By construction of $\mc{P}$ we note the important bounds
\begin{align}\label{smallr}
r(p)a_i(p)/p^{1/2}, |r(p)|^2=o(1)
\end{align}
for any $1\leqs i\leqs m$ and $p\in\mc{P}$. With these bounds the computations for the Euler products acquired from the resonance method can proceed in the usual simple way. This is the reason for the restriction  $|a_i(p)|\leqs (\log p)^{1-\epsilon}$ in $\mc{P}$; without it such computations are much more involved e.g. see \cite{BFKMMS}. 

With this set-up and the above notation we have the following propositions.
 
 \begin{prop}\label{gen L lower prop} For large $T$ and $X=T^\Delta$ with $\Delta<1$, we have
 \[
\frac{1}{T} \int_{T}^{2T}|L(\tfrac12+it)|^2|R(t)|^2dt 
\gg 
\prod_p\bigg(1+|r(p)|^2+2(1+o(1))\frac{{r(p)}a(p)}{\sqrt{p}}\bigg).
 \]
 \end{prop}
 
  \begin{prop}\label{spec upper bound prop} Let $L_i(s)$ be a primitive Dirichlet $L$-function or the $L$-function of a (holomorphic or Maa\ss) cuspidal newform. Then for large $T$ and $X=T^\Delta$ with $\Delta$ sufficiently small, 
 \[
\frac{1}{T} \int_{T}^{2T}|L_i(\tfrac12+it)|^2|R(t)|^2dt 
\ll  
\prod_p\bigg(1+|r(p)|^2+2(1+o(1))\frac{\Re(r(p)\ol{a_i(p)})}{\sqrt{p}}\bigg).
 \]
 If $L_i$ is a Dirichlet $L$-function, one can take $\Delta<\frac{17}{33}$ and if $L_i$ is a $GL(2)$ $L$-function, we can take $\Delta<\frac{1/2-\theta}{3+\theta}$ where $\theta$ is the bounds towards the Ramanujan conjecture for $GL(2)$ Maa\ss\  Forms. 
 \end{prop}
 
 \begin{prop}\label{gen upper bound prop} Let $\pi_i$ be irreducible unitary cuspidal automorphic representations of $GL(d_i)$ over $\mathbb{Q}$ such that $\pi_i\not \cong \pi_j$ for $1\leq i\not=j\leq m$. Assume GRH for each $L_j(s)=L(s, \pi_j), j=1, \dots, m$ and if $d_j\geqs 4$  assume the generalised Ramanujan conjecture for $L_j(s)$. Then for large $T$ and $X=T^\Delta$ with $\Delta<1/2$, we have  
 \begin{align*}
&\frac{1}{T}\int_{T}^{2T}\prod_{\substack{1\leqs j\leqs m\\j\neq i}}|L_j(\tfrac12+it)|^{2}|R(t)|^2dt
\\
\ll
&\exp\bigg(\sqrt{\frac{\log T}{\log_2 T\log_3 T}}\bigg)\prod_p\bigg(1+|r(p)|^2+2(1+o(1))\frac{\Re ( r(p)\ol{b_i(p)})}{\sqrt{p}}\bigg)
 \end{align*}
where 
 \[
 b_i(p)=\sum_{j\neq i}a_j(p)=a(p)-a_i(p).
 \]
 \end{prop}

\begin{rem}The assumption of the generalised Ramanujan conjecture for $d_j\geqs 4$ arises from Lemma \ref{prime sum cor} below. The remainder of the proof of Proposition \ref{gen upper bound prop} only requires the pointwise bounds on the coefficients given in \eqref{a bound} and \eqref{a pp bound}. The generalised Ramanujan conjecture could be replaced with Hypothesis H of Rudnick--Sarnak \cite{RS}, inequality \eqref{H hyp}, but at the cost of a  more technical proof. 
\end{rem}

 With these propositions in hand we can complete the proofs of  Theorems \ref{main thm} and \ref{main thm 2}.

 \begin{proof}[Proof of Theorems \ref{main thm} and \ref{main thm 2}] 
 
 Define $V$ as  
 \[
V=\frac{\int_T^{2T}|L(\tfrac12+it)|^2|R(t)|^2 dt}{2\sum_{i=1}^m\int_T^{2T} \prod_{j\neq i}|L_j(\tfrac12+it)|^2|R(t)|^2dt}.
 \]
 Then there exists $t\in[T,2T]$ satisfying \eqref{gen ineq} so that 
 \begin{align}
 |L_j(\tfrac{1}{2}+it)|^2>V \text{ for all $1\leqs j\leqs m$.}
 \end{align} It remains to give a lower bound for $V$. From Propositions \ref{gen L lower prop}-\ref{gen upper bound prop} and \eqref{smallr}, we have that 
 \[
 V\gg \exp\Big(-c\sqrt{\frac{\log T}{\log_2T\log_3T}}\Big)\exp\bigg(2(1+o(1))\min_{1\leqs i\leqs m}\sum_{p\in\mc{P}}\frac{\mc{L}(|a_i(p)|^2+\Re \ol{a_i(p)}\sum_{j\neq i}a_j(p))}{p\log p}\bigg).
 \]
 
 Now let us remove the restriction on the size of the $|a_i(p)|$. For $1\leq i\not=j\leqs m$,
 \[
\mc{L} \sum_{\substack{\mc{L}^2<p\leqs \exp((\log \mc{L})^2)\\|a_k(p)|>(\log p)^{1-\epsilon} \text{for some } k}}\frac{a_i(p)\ol{a_j(p)}}{p\log p}
\ll
\frac{\mc{L}}{(\log\mc{L})^{2-\epsilon}}
\sum_{\substack{\mc{L}^2<p\leqs \exp((\log \mc{L})^2)}}\frac{|a_i(p){a_j(p)}a_k(p)|}{p}
 \]
 which by H\"older's inequality and \eqref{fourth moment} or the generalised Ramanujan conjecture is 
 \[
 \ll \frac{\mc{L}}{(\log\mc{L})^{2-\epsilon}}\log\log \mc{L}=o\Big(\frac{\mathcal L}{\log \mathcal L}\Big).
 \]
Thus we can extend the sum over $p\in \mc{P}$ to all $\mc{L}^2<p\leqs \exp((\log\mc{L})^2)$ with an acceptable error. Applying Proposition \ref{res sum bound cor} gives 
   \[
 V\gg\exp\bigg(2(1+o(1))\sqrt{\frac{\frac{1}{m}\log X}{\log\log X}}\bigg).
 \] 
 \end{proof}


 \section{Lower bounds in $t$-aspect: Proof of Proposition \ref{gen L lower prop}}
Let
\[
\mc{I}:=\frac{1}{T}\int_{T}^{2T}|L(\tfrac12+it)|^2|R(t)|^2dt.
\]
We aim to prove the lower bound 
\[
\mc{I}
\gg 
\prod_p\bigg(1+|r(p)|^2+2(1+o(1))\frac{\ol{r(p)}a(p)}{\sqrt{p}}\bigg).
\]
Let $\Phi$ be a smooth function supported on $[1,2]$ satisfying $0\leqs \Phi(x)\leqs 1$ so that
\[\mc{I}
\geqs 
\frac{1}{T}\int_\mb{R}|L(\tfrac12+it)|^2|R(t)|^2\Phi(t/T)dt.
\]
Thus, by Cauchy--Schwarz
\begin{equation}\label{initial ineq}
\mc{I}
\geqs 
\bigg|\frac{1}{T}\int_\mb{R}L(\tfrac12+it)|R(t)|^2\Phi(t/T)dt\bigg|^2/\frac{1}{T}\int_\mb{R}|R(t)|^2\Phi(t/T)dt. 
\end{equation}
As usual, since $X=T^\Delta$ with $\Delta<1$, we find  
\begin{equation}\label{res bound}
\frac{1}{T}\int_\mb{R}|R(t)|^2\Phi(t/T)dt 
\sim 
\hat{\Phi}(0)\sum_{n\leqs X} |r(n)|^2 
\leqs 
\hat{\Phi}(0)\prod_p (1+|r(p)|^2) 
\end{equation}
and so it remains to compute 
\[
\frac{1}{T}\int_{\mb{R}}L(\tfrac12+it)|R(t)|^2\Phi(t/T)dt.
\]
This has essentially been done by Aistleitner--Pa\'nkowski \cite{AP}, and so we only present the main details. The only difference is that they worked with the Selberg class where the generalised Ramanujan conjecture is assumed, although this has little effect on the arguments.  

Applying the Mellin inversion formula 
\[
e^{-x}=\frac{1}{2\pi i }\int_{(c)}\Gamma(s) x^{-s}ds,\qquad c>0
\]
and shifting contours to the left along with the convexity bounds \eqref{phrag}, we find 
\[
L(\tfrac12+it)=\sum_{n=1}^\infty \frac{a(n)}{n^{1/2+it}}e^{-n/Y}+O(1)
\]
where $Y=T^{d_L+\epsilon}$ and 
$
d_L=\sum_{i=1}^m d_i
$
 is the degree of $L(s)$. For $n>3Y\log Y$ we have $e^{-n/Y}\leqs n^{-2}$ and hence  
\[
\sum_{n\geqs 3Y\log Y}\frac{|a(n)|e^{-n/Y}}{n^{1/2}}\ll \sum_{n\geqs 1} \frac{\tau_{d_L}(n)}{n^{2+1/(d^2+1)}} \ll 1
\]
where we have used the coefficient bound for $a(n)$ in \eqref{coeff bound}. Thus we find 
\[
L(\tfrac12+it)=\sum_{n\leqs T^{d_L+2\epsilon}} \frac{a(n)}{n^{1/2+it}}e^{-n/Y}+O(1).
\]
From the rapid decay of $\hat{\Phi}$ and \eqref{res bound} we now have 
\begin{multline}\label{lower bound integral}
\frac{1}{T}\int_{\mb{R}}L(\tfrac12+it)|R(t)|^2\Phi(t/T)dt
=
\hat{\Phi}(0)\sum_{lm=n\leqs X}\frac{a(l)r(m)\overline{r(n)}}{\sqrt{l}}e^{-l/Y}
+
O\big(\prod_p(1+|r(p)|^2)\big).
\end{multline}
Since $r(n)$ is supported on squarefree integers and $r(p)=a(p)\mc{L}/p^{1/2}\log p$ where it is non-zero, the summand of the main term is positive and hence we can bound it from below by
\begin{equation}\label{lower bound sum}
\frac12\hat{\Phi}(0)\sum_{lm\leqs X}\frac{a(l)r(m)\overline{r(lm)}}{\sqrt{l}}
\end{equation}
since $e^{-l/Y}\geqs 1/2$ for $l\leqs X$. Next we extend the sum $lm\leq X$ to all $l,m$.
By Rankin's trick 
\begin{equation}
\begin{split}\label{rankin bound}
&
\frac{\sum_{lm>X}{a(l)r(m)\overline{r(lm)}}/{\sqrt{l}}}{\sum_{l,m}{a(l)r(m)\overline{r(lm)}}/{\sqrt{l}}}
\ll
X^{-\alpha}\prod_p \frac{1+|r(p)|^2p^{\alpha}+|r(p)a(p)|p^{-1/2+\alpha}}{|1+|r(p)|^2+a(p)\ol{r(p)}p^{-1/2}|}
\\
\ll &
\exp\bigg(-\alpha\log X+\sum_{L^2<p\leqs \exp((\log L)^2)}|a(p)|^2\Big(\frac{\mc{L}^2}{p\log^2 p}+\frac{\mc{L}}{p\log p}\Big)(p^{\alpha}-1)\bigg)
\end{split}
\end{equation}
for any $\alpha>0$. Applying this along with the bound \eqref{res sum bound} and choosing $\alpha=1/(\log\mc{L})^3$, we find that \eqref{rankin bound} can be bounded by
\begin{equation}\label{Rankin error}
\ll \exp\Big(-\alpha \frac{\log X\log_3X}{\log_2X}+O(\alpha \frac{\mc L}{(\log \mc L)^2}+\alpha^2 \mc L^2 \log\log \mc L)\Big)\ll \exp \Big(- \frac{\log X}{(\log_2 X)^3}\Big).
\end{equation} 
Combining \eqref{lower bound integral}, \eqref{lower bound sum}, \eqref{rankin bound} and \eqref{Rankin error}, we find that 
\[
\frac{1}{T}\int_{\mb{R}}L(\tfrac12+it)|R(t)|^2\Phi(t/T)dt
\gg
\prod_p \bigg(1+|r(p)|^2+\frac{a(p)\ol{r(p)}}{\sqrt{p}}\bigg).
\]
Applying this in \eqref{initial ineq} together with \eqref{smallr} we find that, 
\begin{align*}
\mc{I}\gg \prod_p\frac{\Big(1+|r(p)|^2+\frac{a(p)\ol{r(p)}}{p^{1/2}}\Big)^2}{1+|r(p)|^2} 
\gg
\prod_p \bigg(1+|r(p)|^2+2(1+o(1))\frac{a(p)\ol{r(p)}}{\sqrt{p}}\bigg)
\end{align*}
as desired.


\section{Unconditional upper bounds in $t$-aspect: Proof of Proposition \ref{spec upper bound prop}} 

In this section we give the proof of Proposition \ref{spec upper bound prop} depending on whether $L_i$ is a Dirichlet $L$-function or a $GL(2)$ $L$-function. For this we utilise the requisite twisted moment formulas which are known in these cases. 


\subsection{Dirichlet $L$-functions}  In the case when $L_i(s)=L(s,\chi)$ where $\chi$ is a primitive Dirichlet character modulo $q$ we have the following. 
\begin{lem}\label{twist chi}
	Let $\chi$ be a primitive Dirichlet character modulo $q$. Let $R(t)=\sum_{n\leq X}r(n)n^{-it}$ be as in \eqref{Rtdef}. Let $\alpha, \beta$ be complex numbers such that $\alpha, \beta\ll 1/T$. Then for $X=T^\Delta$ with $\Delta<\frac{17}{33}$ there exists $\epsilon_\Delta>0$ such that 
	\begin{align}
	&\int_{T}^{2T}L(\tfrac{1}{2}+\alpha+it, \chi)L(\tfrac{1}{2}+ \beta-it, \bar{\chi})|R(t)|^2dt+O(T^{1-\epsilon_\Delta})\\=
	&\sum_{(hk,q)=1}\frac{(h,k)^{1+\alpha+\beta}r(h)\ol{\chi(h)}r(k)\chi(k)}{h^{1/2+\beta}k^{1/2+\alpha}}\Big(L(1+\alpha+\beta, \chi_0)+\Big( \frac{qt(h,k)^2}{2\pi hk}\Big)^{-\alpha-\beta}L(1-\alpha-\beta,\chi_0)\Big),
	\end{align}
	where $\chi_0$ is the principal Dirichlet character modulo $q$.
\end{lem}
\begin{proof}
The proof is similar to that \cite[Theorem 1.1]{Wu}, but certain modifications are needed. First note that the condition $a(h)\ll h^\epsilon$ in the assumption \cite[Theorem 1.1]{Wu} does not hold in our case, however, we can modify the proof so that the conclusion still holds. More specifically, we still have \cite[eq (5.28)]{Wu} since $\sum_{u\sim U}\frac{\sqrt{u}r(u)}{u}\ll U\prod_{p}(1+r(p)\sqrt{p})\ll U \exp\Big(O\Big(\frac{\mathcal L}{\log\mathcal L^2}\Big)\Big)\ll U T^\epsilon$ and this causes an extra factor of $T^\epsilon$ in \cite[eq (5.29)]{Wu} which is acceptable. In the estimate \cite[eq (5.37)]{Wu}, we use $\sum_{u\sim U}|\frac{\sqrt{u}r(u)}{u}|^2\ll \prod_{p}\Big(1+\frac{r(p)^2}{p}\Big)\ll \exp\Big(O\Big(\frac{\mathcal L^2}{(\log \mathcal L^2)^2}\Big)\Big)\ll T^\epsilon$, which is again acceptable and leads to \cite[eq (5.38)]{Wu}.  The main term can be derived the same way in the proof of \cite[Theorem 1.1]{Wu} using \cite[Proposition 3.1]{Wu} (after correcting the typo in the exponent of $(h,k)$) or \cite[Lemma 1 and Section 5]{Conrey}.
\end{proof}

Let $X=T^{\Delta}$ with $\Delta< \frac{17}{33}$. We write 
\begin{align}\label{second Dirichlet}
\int_{T}^{2T}|L(\tfrac12+it,\chi)|^2|R(t)|^2dt
=
\lim_{\alpha,\beta\to 0}
\int_T^{2T}L(\tfrac12+\alpha+it,\chi)L(\tfrac12+\beta-it,\ol{\chi})|R(t)|^2dt.
\end{align}
By a residue calculation, we see that
\begin{align*}
&{L(1+\alpha+\beta,\chi_0)}+\bigg(\frac{qt}{2\pi HK}\bigg)^{-\alpha-\beta}{L(1-\alpha-\beta,\chi_0)}
\\
&=
-\frac{(\frac{qt}{2\pi HK})^{-(\alpha+\beta)/2}}{(2\pi i)^2}\int_{|z_j|=\tfrac{2^j}{\log T}} L(1+z_1-z_2,\chi_{0})(z_1-z_2)^2
\bigg(\frac{qt}{2\pi HK}\bigg)^{\tfrac{z_1-z_2}{2}}
\prod_{j=1}^2\frac{dz_j}{(z_j-\alpha)(z_j+\beta)}.
\end{align*}
Applying this in Lemma \ref{twist chi} with $HK=hk/(h,k)^2$, we obtain 
\begin{align}\label{z int}
&\int_T^{2T}|L(\tfrac12+it,\chi)|^2|R(t)|^2dt+O(T^{1-\epsilon_{\Delta}})
\\&=
\frac{1}{(2\pi i)^2}\int_{|z_j|=\tfrac{2^j}{\log T}} L(1+z_1+z_2,\chi_{0})G_X(z_1,z_2)(z_1+z_2)^2
\bigg(\int_T^{2T}\bigg(\frac{qt}{2\pi}\bigg)^{\tfrac{z_1+z_2}{2}}dt\bigg)\prod_{j=1}^2\frac{dz_j}{z_j^2},
\end{align}
where  
\[
G_X(z_1,z_2)=\sum_{h,k\leqs X}\frac{\ol{\chi}(h/(h,k)){\chi}(k/(h,k))r(h)\ol{r(k)}}{(hk)^{1/2+(z_1+z_2)/2}}(h,k)^{1+z_1+z_2}.
\]

By Rankin's trick we see that 
\begin{align}\label{Gz}
G_X(\ul{z})=N_X(\ul{z})+O(\mc{E}_X(\ul{z}))
\end{align}
where 
\begin{align*}
N_X(\ul{z})
= 
\sum_{h,k}\frac{\ol{\chi}(h/(h,k)){\chi}(k/(h,k))r(h)\ol{r(k)}}{(hk)^{(1+z_1+z_2)/2}}(h,k)^{1+z_1+z_2}
= 
\prod_p \bigg(1+|r(p)|^2+2\frac{\Re r(p)\ol{\chi }(p)}{p^{(1+z_1+z_2)/2}}\bigg)
\end{align*}
and 
\[
\mc{E}_X(\ul{z})=X^{-\alpha}\prod_p \bigg(1+|r(p)|^2p^\alpha+\frac{ |r(p)|}{p^{(1+\Re z_1+\Re z_2)/2}}(p^\alpha+1)\bigg)
\]
for any $\alpha>0$.
Using the approximation \eqref{Gz} in \eqref{z int} and calculating the integral of the main term via residues at $z_j=0$, we find that the leading term in \eqref{z int} is of size
\[
T(\log T )N_X(\ul{0}) 
\]
with the lower order terms 
involving partial derivatives of $N_X(\ul{z})$. To estimate these we note 
\begin{equation}\label{partial}
\begin{split}
\frac{\partial}{\partial z_1}N_X(\ul{z})\bigg|_{\ul{z}=\ul{0}}
\ll &
N_X(\ul{0})\sum_p\frac{|r(p)|\log p/p^{1/2}}{|1+|r(p)|^2+2\Re\frac{r(p){\ol{\chi(p)}}}{p^{1/2}}|} 
\\
\ll &
N_X(\ul{0})\sum_{\mc{L}^2<p\leqs \exp((\log\mc{L})^2)}\frac{\mc{L}|a(p)|}{p}
\ll
N_X(\ul{0})(\log X)^{1/2+\epsilon}
\end{split}
\end{equation}
by Cauchy--Schwarz, \eqref{smallr} and \eqref{Selberg conj 1}.
Note that the integrand of \eqref{z int} is $\ll (\log T)^3$ and so trivial estimation of the contribution from $\mc{E}_X(\ul{z})$ to this integral gives
\[
\mc{J}_i\leqs C(\log T) N_X(\ul{0}) +O((\log T)\mc{E}(X))+O(T^{-\epsilon_\Delta}),
\] 
where $\mc{E}(X)=\mc{E}_X(-2/\log T,-4/\log T)$. Now
\begin{equation}\label{ratio bound}
\begin{split}
\frac{\mc{E}(X)}{N_X(\ul{0})}
\ll &
\exp\Big(-a\log X + \sum_{p} |r(p)|^2(p^\alpha-1)+O(\sum_p|r(p)|p^{-1/2})\Big)
\\
\leqs &
\exp\Big(-a\log X + \sum_{\mc{L}^2<p\leqs \exp((\log \mc{L})^2)} (p^\alpha-1)\frac{\mc{L}^2|a(p)|^2}{p\log^2 p}+O(\frac{\mc{L}}{\log\mc{L}})\Big)
\end{split}
\end{equation}
which, similarly to \eqref{rankin bound}, is $o(1)$ on choosing $\alpha=1/(\log \mc{L})^3$. 
Thus, we find that when $L_i(s)=L(s,\chi)$, 
\[
\mc{J}_i\ll (\log T) N_X(\ul{0}) \ll \log T\prod_p  \bigg(1+|r(p)|^2+2\Re\frac{ r(p)\ol{a_i}(p)}{p^{1/2}}\bigg),
\]
which completes the proof of Proposition \ref{spec upper bound prop} for the case for Dirichlet $L$-functions after noting that the factor of $\log T$ can be absorbed into the $o(1)$ term in the product.


\subsection{$GL(2)$ $L$-functions} Here we are in the case where $L_i=L(s,f)$ is the $L$-function of a primitive cusp form $f$. 
Let $\Phi: \mathbb{R}\rightarrow \mathbb{R}$ be a smooth function supported on $[1/4, 2]$ satisfying $\Phi(x)\geqs 1$ on $x\in [1,2]$ along with the bounds $\Phi^{(j)}(x)\ll (\log T)^{j}$ for each $j\geq 0$. 
 
\begin{lem}\label{twisted 2nd mmt thm 2}Let $L(s,f)=\sum_{n\geqs 1}\lambda_f(n)n^{-s}$ be the $L$-function of a Hecke newform (holomorphic or Maa\ss) of level $N$. Let $\alpha, \beta$ be complex numbers satisfying $\alpha, \beta \ll 1/\log T$. Let $(h,k)=1$ and $\Phi$ be as above. Then
 we have
\begin{align*}
&\int_\mb{R}L(\tfrac12+\alpha+it,f)L(\tfrac12+\beta-it,{f})(h/k)^{-it}\Phi\big(\frac{t}{T}\big)dt
\\
= &
\frac{1}{{h^{1/2+\beta}k^{1/2+\alpha}}}\int_\mb{R} \Big(L^*(1+\alpha+\beta,f\otimes f)Z_{\alpha,\beta}(h,k)
\\
&
+
\bigg(\frac{t\sqrt{N}}{2\pi \sqrt{hk}}\bigg)^{-2(\alpha+\beta)}L^*(1-\alpha-\beta,f\otimes  f)Z_{-\beta,-\alpha}(h,k)\Big)\Phi\big(\frac{t}{T}\big)dt
+O((hk)^{1/2+\epsilon}T^{1/2+\theta+\epsilon})
\end{align*}
where $L^*(s, f\otimes f)=\sum_{n}\lambda_f(n)^2 n^{-s}$ and 
\begin{align}\label{Zdef}
Z_{\alpha, \beta}(h,k)=\prod_{p\mid hk}(1-p^{-2(1+\alpha+\beta)})^{-1}\Big(\lambda_f(p)-\frac{\lambda_f(p)}{p^{1+\alpha+\beta}}\Big).
\end{align} 
\end{lem}
\begin{proof} 
This follows from \cite{AT}, which improves earlier results for the holomorphic case in \cite{Ber, KRZ}. 
Using \cite[Proposition 3.4]{AT}, we have for $|\alpha+ \beta|\gg 1/\log T$
\begin{align*}
&\int_\mb{R}L(\tfrac12+\alpha+it,f)L(\tfrac12+\beta-it,{f})(h/k)^{-it}\Phi(t/T)dt
+O((hk)^{1/2}T^{1/2+\theta+\epsilon})
\\
= &
\sum_{hm=kn}\frac{\lambda_f(m)\lambda_f(n)}{m^{1/2+\alpha}n^{1/2+\beta}}\int_\mb{R}V_{\alpha,\beta}(mn,t)\Phi(t/T)dt
\\
&
+
\sum_{hm=kn}\frac{\lambda_f(m)\lambda_f(n)}{m^{1/2-\beta}n^{1/2-\alpha}}\int_\mb{R}X_{\alpha,\beta,t}V_{-\beta,-\alpha}(mn,t)\Phi(t/T)dt
\end{align*}
where 
\[
V_{\alpha,\beta}(x)
=
\frac{1}{2\pi i }\int_{1-i\infty}^{1+i\infty} \frac{G(s)}{s}x^{-s} g_{\alpha,\beta}(s,t)ds
\]
with \[G(s)=
e^{s^2}\frac{(\alpha+\beta)^2-(2s)^2}{(\alpha+\beta)^2},
\]  and where $g_{\alpha,\beta}(s,t)$ and $X_{\alpha,\beta,t}$ are ratios of gamma factors satisfying 
\begin{equation}\label{g bound}
g_{\alpha,\beta}(s,t)
= 
\bigg(\frac{t\sqrt{N}}{2\pi}\bigg)^{2s}\Big(1+O\big(\frac{|s|^2}{t}\big)\Big),\quad
X_{\alpha,\beta,t}
= 
\bigg(\frac{t\sqrt{N}}{2\pi}\bigg)^{-2(\alpha+\beta)}\Big(1+O\big(\frac{|\alpha^2-\beta^2|}{t}\big)\Big)
\end{equation}
(see e.g. Lemma 2 of \cite{Ber}).
Using the definition of $V_{\alpha,\beta}(x)$ and moving the $m,n$-sum inside, we encounter the Dirichlet series 
\begin{align*}
\sum_{hm=kn}\frac{\lambda_f(m)\lambda_f(n)}{m^{1/2+\alpha+s}n^{1/2+\beta+s}}
= &
\frac{1}{k^{1/2+\alpha+s}h^{1/2+\beta+s}}\sum_{l\geqs 1}\frac{\lambda_f(kl)\lambda_f(hl)}{l^{1+\alpha+\beta+2s}}
\end{align*}
since $(h,k)=1$. Using multiplicativity and Hecke relations (see e.g. \cite[Proof of Lemma 7.9]{BFKMMS}), we see that 
\begin{align}
D(s;h,k):=\sum_{l\geq 1}\frac{\lambda_f(kl)\lambda_f(hk)}{l^s}=L^*(s,f\otimes f)\prod_{p\mid hk}(1-p^{-2s})^{-1}\Big(\lambda_f(p)-\frac{\lambda_f(p)}{p^s}\Big)
\end{align}  
where $L^*(s,f\otimes f)=\sum_{n\geq 1}\lambda_f(n)^2n^{-s}$. 
Shifting the contour to $\Re(s)=-1/4+\epsilon$ we encounter a simple pole at $s=0$ 
which gives the main term. The contribution from the remaining contour is seen to be $\ll T^{1/2}(hk)^{-1/4+\theta+\epsilon}$ by \eqref{g bound}, the rapid decay of $G(s)$, the convexity bound $L^*(1/2+\epsilon+iy,f\otimes f)\ll (1+|y|)^{1+\epsilon}$, and the bound
\[
 \prod_{p\mid hk}\Big(\lambda_f(p)+O(\frac{\lambda_f(p)}{p^{1/2+2\epsilon}})\Big) \ll (hk)^{\theta+\epsilon}.
\]
By analytic continuation, the result hold for $\alpha, \beta\ll 1/\log T$. 
\end{proof}

 To complete the proof of Proposition \ref{spec upper bound prop} for the case of $GL(2)$ $L$-functions, we follow the same argument as before in the case for Dirichlet $L$-functions after replacing Lemma \ref{twist chi} by Lemma \ref{twisted 2nd mmt thm 2} and $G_X(z_1, z_2)$ by
\[
H_X(z_1,z_2)
=
\sum_{h,k\leqs X}\frac{r(h)\ol{r(k)}Z_{z_1,z_2}(h/(h,k),k/(h,k))}{(hk)^{(1+z_1+z_2)/2}}(h,k)^{1+z_1+z_2}
\] 
where $Z_{z_1, z_2}$ is defined in \eqref{Zdef}. Note that we have 
\begin{align}
Z_{\ul{0}}(h,k)=\prod_{p\mid hk}(1-p^{-2})^{-1}\Big(\lambda_f(p)-\frac{\lambda_f(p)}{p}\Big)
\end{align}
and that $\lambda_f(p)(1-1/p)=a_i(p)(1+o(1))$ for large $p$.
Thus, the main contribution to $\mc{J}_i$ in this case, aside from some factors of $\log T$ which can be absorbed into the $o(1)$, is
\[
\sum_{h,k}\frac{r(h)\ol{r(k)}Z_{\ul{0}}(h/(h,k),k/(h,k))}{(hk)^{1/2}}(h,k)
=
\prod_p\bigg(1+|r(p)|^2+2(1+o(1))\Re\frac{r(p)\ol{a_i(p)}}{p^{1/2}}\bigg)
\]
as required. Again, the lower order terms coming from partial derivatives can be bounded similarly to \eqref{partial} using Proposition \ref{res sum bound cor} whilst the error from the tail sums $h>X$, $k>X$ are also of a lower order by similar arguments to before (again using Proposition \ref{res sum bound cor}).


\section{Conditional upper bounds in $t$-aspect: Proof of Proposition \ref{gen upper bound prop}}\label{cond upper sec} 

\subsection{Upper bounds for the logarithm of the product of $L$-functions} Let $\pi_i$ be irreducible unitary cuspidal automorphic representations of $GL(d_i)$ over $\mathbb{Q}$ such that $\pi_i\not \cong \pi_j$ for $1\leq i\not=j\leq m$. For a given $1\leqs i\leqs m$ write 
\[
M(s)=M_i(s)=\prod_{\substack{j=1\\j\neq i}}^m L_j(s)=\sum_{n=1}^\infty \frac{B_i(n)}{n^s}
\]
so that \[\frac{M'(s)}{M(s)}=\sum_{n}\frac{\Lambda(n)b(n)}{n^s}\] where
\[
b(n)=b_i(n)=\sum_{\substack{j=1\\j\neq i}}^ma_{\pi_j}(n)
\]
where $a_{\pi}(n)$ are defined as in \eqref{apidef}. 
The goal is to show that 
 \begin{align*}
\frac{1}{T}\int_{T}^{2T}|M(\tfrac12+it)|^{2}|R(t)|^2dt
\ll
\exp\bigg(\sqrt{\frac{\log T}{\log_2 T\log_3 T}}\bigg)\prod_p\bigg(1+|r(p)|^2+2\Re\frac{r(p)\ol{b(p)}}{p^{1/2}}\bigg).
 \end{align*}
We plan to compute the integral over $t$ by using Harper's method \cite{Ha}. A key estimate will be
\begin{equation}\label{b prime sum}
\sum_{p\leqs x}\frac{|b(p)|^2}{p}=(m-1)\log\log x+O(1)
\end{equation}
which follows by \eqref{Selberg conj 1} and the assumption on $\pi_j$. 
To measure the size of exceptional sets where Dirichlet polynomials obtain large values we need the following standard lemma. 
\begin{lem}[\cite{Sound}, Lemma 3]\label{mont lem}Let $T$ be large and let $2\leqs x\leqs T$. Let $k$ be a natural number such that $x^{k}\leqs T$. Then for any complex numbers $c(p)$ we have 
\[
\frac{1}{T}\int_{T}^{2T}\bigg|\sum_{p\leqs x}\frac{c(p)}{p^{1/2+it}}\bigg|^{2k}dt\ll k!\bigg(\sum_{p\leqs x}\frac{|c(p)|^2}{p}\bigg)^k.
\]
\end{lem}

To obtain an upper bound for $|M|^2$ we apply the following generalisation of Soundararajan's result for the Riemann zeta function \cite[Proposition]{Sound} due to Chandee \cite{Chandee}. 
\begin{lem}[\cite{Chandee}, Theorem 2.1]\label{sound ineq lem}
Assume GRH holds for $L(s,\pi_j)$, $1\leqs j\leqs m$, $j\neq i$. Then for $2\leqs Z\leqs T^2$ and $t\in[T,2T]$ we have 
\begin{equation}\label{sound ineq 0}
\log|M(1/2+it)|\leqs \Re\sum_{n\leqs Z}\frac{\Lambda(n)b(n)w_Z(n)}{n^{1/2+it}\log n}+C_M\frac{\log T}{\log Z}+O(1)
\end{equation}
for some positive constant $C_M$ where 
\[
w_Z(n)=n^{-1/2\log Z}\Big(1-\frac{\log n}{\log Z}\Big)\leqs 1.
\]
\end{lem}

Since we are interested in extreme values, powers of $\log T$ will not meaningfully affect our final bound and  so we first trivially bound the sum over prime powers. 

\begin{lem}\label{prime sum cor}
Assume GRH holds for $L(s,\pi_j)$, $1\leqs j\leqs m$, $j\neq i$ and that the generalised Ramanujan conjecture holds for $L(s,\pi_j)$ if $d_j\geqs 4$. Then for $2\leqs Z\leqs T^2$ and $t\in[T,2T]$ we have 
\begin{equation}\label{sound ineq}
\log|M(1/2+it)|\leqs \Re\sum_{\substack{p \leqs Z}}\frac{b(p)w_Z(p)}{p^{1/2+it}}+C_M\frac{\log T}{\log Z}+O(\log\log Z).
\end{equation}
\end{lem}
\begin{proof}
	The terms in Lemma \ref{sound ineq lem} with powers $\ell>d^2+1$ with $d=\max_j d_j$ are 
\begin{align}
\ll \sum_{\substack{p^\ell \leqs Z\\ \ell > d^2+1}}\frac{|b(p^\ell)|}{p^{\ell/2}}\ll \sum_{p}\frac{1}{p^{\ell/(d^2+1)}}\ll 1
\end{align}
which follows from \eqref{a bound} and the fact that $b(p^\ell)=\sum_{j\neq i}a_{j}(p^\ell)$. For $2\leqs \ell \leqs d^2+1$ we note that if $d_j\leqs 3$ then \eqref{a pp bound} gives
\[
|a_j(p^\ell)|\ll 1+ |a_j(p)|^\ell.
\] 
If $d_j\geqs 4$ then the generalised Ramanujan conjecture implies $|a_j(p^\ell)|\leqs d_j$. 
Thus the prime squares contribute
\[
\ll \sum_{p\leqs Z^{1/2}}\frac{|b(p^2)|}{p} 
\ll \log\log Z
\]
by \eqref{Selberg conj 1} whilst by \eqref{a bound} we have 
\[
\sum_{p}\frac{|b(p)|^\ell}{p^{\ell/2}}\ll \sum_{j\neq i}\sum_{p}\frac{|a_j(p)|^2}{p^{1+(\ell-2)/(d_j^2+1)}}.
\]
Since the Rankin--Selberg $L$-function is convergent for $\sigma>1$, this last sum is bounded for $\ell\geqs 3$.

\end{proof}

\subsection{Initial splitting and the exceptional set} 
Following \cite{Ha}, the aim is to take a reasonably large $Z$ in \eqref{sound ineq} and split the sum over primes into pieces with small variance so that for typical $t$ their exponential can be approximated by a short truncated Taylor series. For us, the choice at where to begin this splitting is dictated by the support of the resonator coefficients, namely $p\leqs \exp((\log\mc{L})^2)$ - the main interaction between $M$ and the resonator will come from this piece. This gives a large chunk of primes in the first sum but as the following lemma shows, this is just about affordable and the exceptional set of large values of this sum is sufficiently small in measure.      
\begin{lem}\label{excep lem}
For $Z\leqs X$  let 
\[
E=\Big\{t\in [T,2T]: \Big|\sum_{\substack{p\leqs \exp((\log \mc{L})^2)}}\frac{b(p)w_Z(p)}{ p^{1/2+it}}\Big|\geqs \frac{\log T}{100(\log\mc{L})^2}\Big\}.
\]
Then 
\[\mu(E)\ll T\exp\Big(-4(1-o(1))\frac{\log T}{\log_2 T}\Big).\]
In particular, under the assumptions in Proposition \ref{gen upper bound prop} for $X= T^{\Delta}$ with $\Delta<1/2$, we have 
\[
\int_{E}|M(\tfrac12+it)|^2|R(t)|^2dt=o(T).
\]
\end{lem}
\begin{proof}
By Lemma \ref{mont lem} along with \eqref{b prime sum} we have 
\begin{align*}
\frac{1}{T}\mu({E})
\ll &
 k!\bigg(\sum_{\substack{p\leqs \exp((\log\mc{L})^2)}}\frac{|b(p)|^2}{ p}\bigg)^k \bigg(\frac{\log T}{100(\log\mc{L})^2}\bigg)^{-2k}
 \\
 \ll & k^{1/2}\bigg(\frac{Ck\log_3 T}{(\log T/(\log \mc{L})^2)^2}\bigg)^k
\end{align*}
for some $C$ provided $k\leqs \frac{\log T}{(\log\mc{L})^2}$. Choosing $k=\frac{\log T}{(\log\mc{L})^2}=4(1+o(1))\frac{\log T}{(\log_2 T)^2}$ this is 
\[
\ll 
\bigg(\frac{C^\prime(\log_2 T)^2\log_3 T}{\log T}\bigg)^{4(1+o(1))\log T/(\log_2 T)^2}
\leqs 
\exp\big(-4(1-o(1))\frac{\log T}{\log_2T}\big)
\]
giving the first part of the lemma. 

By H\"older's inequality we have 
\begin{align*}
\int_{{E}} |M(\tfrac12+it)|^2 |R(t)|^2dt
\leqs &
\mu(E)^{1/4}
 \bigg(\int_{T}^{2T}|M(\tfrac12+it)|^{8}dt\bigg)^{1/4}
 \bigg(\int_{T}^{2T}|R(t)|^4dt\bigg)^{1/2}
\\
\ll &
T^{1/2}\exp\bigg(-(1+o(1))\frac{\log T}{\log_2 T}\bigg)\bigg(\int_{T}^{2T}|R(t)|^4dt\bigg)^{1/2}
\end{align*}
on applying the conditional bound $\int_{T}^{2T} |M(1/2+it)|^8dt\ll T(\log T)^{O(1)}$ which follows from \cite{MTB}.

By the mean value theorems for Dirichlet polynomials, for $X\leqs T^{1/2-\epsilon}$ we have 
\begin{align*}
\int_{T}^{2T}|R(t)|^4dt
\ll &
T\sum_{\substack{n_1n_2=n_3n_4\\n_j\leqs X}}|r(n_1)r(n_2)r(n_3)r(n_4)|
\leqs 
T\prod_p \big(1+4|r(p)|^2+|r(p)|^4\big)
\\
\leqs &
T\exp\bigg(4\sum_{p}|r(p)|^2
\bigg)
\ll 
T\exp\bigg(4\sum_{\mc{L}^2<p}\frac{\mc{L}^2|a(p)|^2}{p\log^2 p}
\bigg)
\\
\ll &
T\exp\bigg(4(1+o(1))\frac{\log X}{\log_2 X}\bigg)\ll T\exp\Big((2-\epsilon+o(1))\frac{\log T}{\log_2T}\Big)
\end{align*}
by \eqref{res sum bound}. 
\end{proof}

\subsection{Remaining splittings and an inequality for $|M|^2$} A key point is that on the set $[T,2T]\backslash E$ the exponential of the sum in the above lemma can be approximated by a short truncated Taylor series to give a Dirichlet polynomial of length $\leqs T^{1/10}$. Our choice of parameters throughout will be dictated by the need to have short Dirichlet polynomials whilst also having small exceptional sets. 

For integer $\mf{i}\geqs 0$ let 
\[
Z_\mf{i}=\exp(e^\mf{i}(\log\mc{L})^2),\qquad \qquad Z_{-1}=1.
\]
Let $J$ be the minimal integer such that $Z_J\geqs \exp({2}{C_M}\sqrt{\log T\log_2 T\log_3 T})$, so that $
J=(\tfrac12+o(1))\log\log T$
and note
\[
C_M\frac{\log T}{\log Z_J}
\leqs 
\frac{1}{2}\sqrt{\log T/\log_2 T\log_3 T}.
\]
By a slight abuse of notation we write
$w_{Z_\mf{j}}(p)$ as $w_\mf{j}(p)$.
Let
\[
P_{\mf{i},\mf{j}}(t)=\sum_{Z_{\mf{i}-1}<p\leqs Z_\mf{i}}\frac{b(p)w_\mf{j}(p)}{p^{1/2+it}}.
\]
so that 
\[
\sum_{p\leqs Z_\mf{j}}\frac{b(p)w_\mf{j}(p)}{p^{1/2+it}}=\sum_{i=0}^\mf{j} P_{\mf{i},\mf{j}}(t).
\]
Set
\[
\ell_\mf{i}=\frac{\log T}{100(\log \mc{L})^2}e^{-5\mf{i}/4},\qquad r_\mf{i}=\frac{\log T}{e^\mf{i}(\log\mc{L})^{2+\epsilon}}.
\]

We remark that $P_{\mf{i},\mf{j}}(t)^{10\ell_\mf{i}}$ is a Dirichlet polynomial of length $Z_{\mf{i}}^{10\ell_{\mf{i}}}=T^{e^{-\mf{i}/4}/10}$. By Lemma \ref{mont lem} and \eqref{b prime sum} the measure of the set where $|P_{\mf{i},\mf{j}}(t)|\geqs \ell_\mf{i}$ is, for any $k\leqs \log T/(e^\mf{i}(\log \mc{L})^2)$,
\begin{equation}\label{excep set P_j}
\ll k^{1/2}\bigg(\frac{k\sum_{Z_{\mf{i}-1}<p\leqs Z_\mf{i}}|b(p)|^2/p}{e\ell_\mf{i}^2}\bigg)^k
\ll
 \bigg(\frac{ck}{\ell_\mf{i}^2}\bigg)^k
 \ll 
 \exp\Big(-\tfrac{c^\prime\log T}{e^\mf{i}(\log_2 T)^{1+\epsilon}}\Big)
\end{equation}
for some constants $c,c^\prime$ on choosing $k=r_\mf{i}$ (we don't choose $k$ as large as possible because we will need shorter Dirichlet polynomials later and this bound is sufficient).
Note this bound kills
\begin{equation}\label{extra term bound}
\exp\Big(C_M\frac{\log T}{\log Z_{\mf{i}-1}}\Big)=\exp\Big((1+o(1))\frac{4C_M\log T}{e^{\mf{i}-1}(\log_2 T)^2}\Big)
\end{equation}
which will be the extra term acquired from applying the inequality \eqref{sound ineq} at the $\mf{i}$th step.  
As a final remark on our parameter choices: the reason for factor $e^{-5\mf{i}/4}$ in $\ell_{\mf{i}}$ is, first of all, so that we have a polynomial of length $T^{e^{-\mf{i}/4}/10}$ which, after taking the product over all $\mf{i}$, is still short (see \eqref{sum length} below). A factor of $e^{-c\mf{i}}$, $c>1$, is required to have the decay in the exponent of $T$ however if $c>3/2$ then $\ell_J$ would not be large enough to guarantee \eqref{excep set P_j}.
The reason for the factor of $e^{-\mf{i}}$ in $r_\mf{i}$ is so that \eqref{excep set P_j} is comparable with \eqref{extra term bound} for all $\mf{i}$. 

Now, by Stirling's formula for $|z|\leqs L$ we have
\begin{equation*}\label{stirling}
e^z=(1+O(e^{-9L}))\sum_{m\leqs 10L}\frac{z^m}{m!}.
\end{equation*}
Therefore, if $|P_{\mf{i},\mf{j}}(t)|\leqs \ell_i$ we have 
\[
\exp(P_{\mf{i},\mf{j}}(t))=(1+O(e^{-9\ell_\mf{i}}))\sum_{m\leqs 10\ell_\mf{i}}\frac{P_{\mf{i},\mf{j}}(t)^m}{m!}. 
\]
The multinomial theorem gives 
\[
P_{\mf{i},\mf{j}}(t)^m=m!\sum_{\substack{\Omega(n)=m\\ p|n\implies Z_{\mf{i}-1}<p\leqs Z_\mf{i}}}\frac{c(n)W_\mf{j}(n)\mf{g}(n)}{n^{1/2+it}}
\]
where 
\begin{equation}\label{coeffs 1}
c(n)=\prod_{p^{\alpha_p}||n}b(p)^{\alpha_p},\qquad W_\mf{j}(n)=\prod_{p^{\alpha_p}||n}w_\mf{j}(p)^{\alpha_p}
\end{equation}
are the completely multiplicative extensions of $b(p)$ and $w_\mf{j}(p)$ to the integers and $\mf{g}$ is the multiplicative function for which 
\begin{equation}\label{coeffs 2}
\mf{g}(p^\alpha)=\frac{1}{\alpha!}.
\end{equation}
Thus, if we denote
\[
\mathcal{N}_{\mf{i},\mf{j}}(t)=\sum_{\substack{\Omega(n)\leqs 10\ell_\mf{i}\\ p|n\implies Z_{\mf{i}-1}< p\leqs Z_\mf{i}}}\frac{c(n)W_\mf{j}(n)\mf{g}(n)}{n^{1/2+it}}
\]
then on such a set of $t$ we have 
\[
\exp\big(2\Re P_{\mf{i},\mf{j}}(t)\big)=(1+O(e^{-9\ell_\mf{i}}))|\mc{N}_{\mf{i},\mf{j}}(t)|^2.
\] 

Accordingly, if $t$ is such that $|{P}_{\mf{i},\mf{j}}(t)|\leqs \ell_\mf{i}$ for all $0\leqs \mf{i}\leqs \mf{j}$ then 
\begin{equation}\label{prod}
\exp\Big(2\Re\sum_{p\leqs Z_\mf{j}}\frac{w_\mf{j}(p)}{p^{{1}/{2}+it}}\Big)
=(1+o(1))
\prod_{\mf{i}=0}^\mf{j}\big|\mathcal{N}_{\mf{i},\mf{j}}(t)\big|^2
\end{equation}
since $\sum_{\mf{i}=0}^\mf{j}e^{-9\ell_\mf{i}}=o(1)$. We note that the right hand side is a Dirichlet polynomial of length
\begin{equation}\label{sum length}
\leqs \prod_{\mf{i}=0}^JZ_\mf{i}^{10\ell_\mf{i}}
= 
T^{\tfrac{1}{10}\sum_{\mf{i}=0}^Je^{-\mf{i}/4}}
\leqs
T^{1/2}
\end{equation}
We can now state an upper bound for the $|M(\tfrac12+it)|$ in terms of these short Dirichlet polynomials.

\begin{lem}\label{zeta lem} Assume GRH for $L(s,\pi_j)$ for $1\leqs j\leqs m$ and let $t\in[T,2T]$. Then either 
\[
|{P}_{0,\mf{j}}(t)|> \ell_0
\] 
for some $0\leqs \mf{j}\leqs J$ or 
\begin{multline*}
|M(\tfrac{1}{2}+it)|^{2}
\ll
\exp\bigg(\frac{1}{2}\sqrt{\frac{\log T}{\log_2 T\log_3 T}}\bigg)\prod_{\mf{i}=0}^J \big|\mathcal{N}_{\mf{i},J}(t)\big|^2
\\
+
(\log T)^{O(1)}\sum_{\substack{ 0\leqs \mf{j}\leqs J-1\\ \mf{j}+1\leqs l\leqs J}} \exp\Big(\frac{C_M\log T}{\log Z_{\mf{j}}}\Big)
 \bigg(\frac{|{P}_{\mf{j}+1,l}(t)|}{\ell_{\mf{j}+1}}\bigg)^{2r_\mf{j}}\prod_{\mf{i}=0}^\mf{j} \big|\mathcal{N}_{\mf{i},\mf{j}}(t)\big|^2.
  \end{multline*}
\end{lem}
\begin{proof}
Suppose $|{P}_{0,\mf{j}}(t)|<\ell_0$. For $0\leqs\mf{ j}\leqs J-1$ let 
\[
S(j)=\left\{t\in[T,2T]: 
\begin{array}{lr}
&|{P}_{\mf{i},l}(t)|\leqs \ell_\mf{i}\,\,\,\,\,\,\,\,\,\,\,\forall 1\leqs \mf{i}\leqs \mf{j},\,\, \forall \mf{j}\leqs l\leqs J;\\
&|{P}_{\mf{j}+1,l}(t)|>\ell_{\mf{j}+1} \text{ for some } \mf{j}+1\leqs l\leqs J
\end{array}
\right\}
\]
and 
\[
S(J)=\bigg\{t\in[T,2T]: |{P}_{\mf{i},J}(t)|\leqs \ell_\mf{i}\,\,\,\,\,\,\,\forall 1\leqs \mf{i}\leqs J\bigg\}.
\]
Then since $[T,2T]=\cup_{\mf{j}=0}^J S(\mf{j})$,  for $t\in[T,2T]$ we have
\begin{equation}\label{zeta decomp}
|M(\tfrac{1}{2}+it)|^{2}
\leqs 
\mathds{1}_{t\in S(J)}\cdot|M(\tfrac{1}{2}+it)|^{2}+\sum_{\substack{0\leqs \mf{j}\leqs J-1\\\mf{j}+1\leqs l\leqs J}}\mathds{1}_{t\in S_l(\mf{j})}\cdot|M(\tfrac{1}{2}+it)|^{2}
\end{equation}
where 
\[
S_l(\mf{j})=\left\{t\in[T,2T]: 
\begin{array}{lr}
|{P}_{\mf{i},l}(t)|\leqs \ell_\mf{i} &\,\,\,\,\,\,\,\,\,\,\,\forall 1\leqs \mf{i}\leqs j,\,\, \forall \mf{j}\leqs l\leqs J;\\
|{P}_{\mf{j}+1,l}(t)|>\ell_{\mf{j}+1} &
\end{array}
\right\}.
\]

We apply Corollary \ref{prime sum cor} to each $M$ on the right hand side of \eqref{zeta decomp}. If $t\in S_l(\mf{j})$ then we take $Z=Z_\mf{j}$ to give
\[
|M(\tfrac{1}{2}+it)|^{2}
\ll
\exp\bigg(2\Re\sum_{p\leqs Z_\mf{j}}\frac{b(p)w_\mf{j}(p)}{p^{1/2+it}}
+\frac{C_M\log T}{\log Z_\mf{j}}
+
O(\log_2 T)\bigg).
\]
For the first sum over primes in the exponential we apply \eqref{prod}. To capture the small size of the set, we multiply by 
\[
 \bigg(\frac{|{P}_{\mf{j}+1,l}(t)|}{\ell_{\mf{j}+1}}\bigg)^{2r_\mf{j}}>1.
 \] 
 If $t\in S(J)$ then we omit this last step. 
\end{proof}

\subsection{Applying the inequality} We apply Lemma \ref{zeta lem} to compute
\[
\int_T^{2T}|M(\tfrac12+it)|^2|R(t)|^2dt
\]
 By Lemma \ref{excep lem} we may disregard the $t$ for which $|P_{0,\mf{j}}(t)|>\ell_0$ since this gives a contribution $o(T)$. By Lemma \ref{zeta lem} the integral over the remaining set is then 
 \begin{multline}\label{inner int ineq}
 \ll \int_T^{2T}\bigg(
\exp\bigg(\frac{1}{2}\sqrt{\frac{\log T}{\log_2 T\log_3 T}}\bigg)\prod_{\mf{i}=0}^J \big|\mathcal{N}_{\mf{i},J}(t)\big|^2
+
(\log T)^{O(1)}
\\
\times\sum_{\substack{ 0\leqs \mf{j}\leqs J-1\\ \mf{j}+1\leqs l\leqs J}} \exp\Big(\frac{C_M\log T}{\log Z_\mf{j}}\Big)
 \bigg(\frac{|{P}_{\mf{j}+1,l}(t)|}{\ell_{\mf{j}+1}}\bigg)^{2r_\mf{j}}\prod_{\mf{i}=0}^\mf{j} \big|\mathcal{N}_{\mf{i},\mf{j}}(t)\big|^2\bigg)|R(t)|^2dt.
 \end{multline}

To facilitate computations we note the following general observations. Suppose we are given $R$ sets $\mc{S}_j\subset\mb{N}$ and  Dirichlet polynomials 
\[
A_j(s) = \sum_{n\in {\mathcal S}_j} a_j(n) n^{-s}, 
\]
where the $\prod_{j=1}^{R} n_j\leqs M=o(T)$ for all $n_j \in {\mathcal S}_j$. 
Then by the mean value theorems for Dirichlet polynomials we have 
\begin{align*}
\frac{1}{T}\int_T^{2T}\prod_{j=1}^R \big|A_j(it)\big|^2dt
\sim&
\sum_{n\leqs M}\Big|\sum_{\substack{n=n_1\cdots n_R\\n_j\in \mc{S}_j}}a_1(n_1)\cdots a_R(n_R)\Big|^2 
\end{align*}
If for any $j_1,j_2$ with $j_1\neq j_2$ the elements of $\mc{S}_{j_1}$ are all coprime to the elements of $\mc{S}_{j_2}$ then
there is at most one way to write $n= \prod_{j=1}^{R} n_j$ with $n_j \in {\mathcal S}_j$ and so
\begin{align*} 
\label{4.6}
\frac{1}{T} \int_T^{2T} \prod_{j=1}^{R} |A_j(it)|^2 dt 
&= (1+O(NT^{-1})) \sum_{n\le N} \Big| \sum_{\substack{ n= n_1 \cdots n_R \\ n_j\in {\mathcal S}_j} }\prod_{j=1}^{R} a_j(n_j) \Big|^2  \nonumber
\\
&= (1+O(NT^{-1})) \prod_{j=1}^R \Big( \sum_{n_j \in {\mathcal S}_j} |a_j(n_j)|^2 \Big) \nonumber \\ 
&= (1+ O(NT^{-1} ))^{1-R} \prod_{j=1}^{R} \Big( \frac 1T \int_{T}^{2T} |A_j(it)|^2 dt \Big). 
\end{align*}

Since $\prod_{\mf{i}=0}^JN_{\mf{i}, \mf{j}}(t)$ is a Dirichlet polynomial of length $\leqs T^{1/2}$ by \eqref{sum length} and $R(t)$ is a Dirichlet polynomial of length $X=T^\Delta$ with $\Delta<1/2$, we can apply the above observations so that \eqref{inner int ineq} becomes 
 \begin{multline*}
 \ll T\exp\bigg(\frac{1}{2}\sqrt{\frac{\log T}{\log_2 T\log_3 T}}\bigg)
  \int |\mc{N}_{0,J}(t)|^2|R(t)|^2dt 
\cdot
\prod_{\mf{i}=1}^J \int\big|\mathcal{N}_{\mf{i},J}(t)\big|^2dt
\\
+
T(\log T)^{O(1)}
\sum_{\substack{ 0\leqs \mf{j}\leqs J-1\\ \mf{j}+1\leqs l\leqs J}} \exp\Big(\frac{C_M\log T}{\log Z_\mf{j}}\Big)
 \int |\mc{N}_{0,\mf{j}}(t)|^2|R(t)|^2dt 
\\
\times 
\prod_{\mf{i}=1}^\mf{j} \int\big|\mathcal{N}_{\mf{i},\mf{j}}(t)\big|^2dt
\int \bigg(\frac{|{P}_{\mf{j}+1,l}(t)|}{\ell_{\mf{j}+1}}\bigg)^{2r_\mf{j}} dt
 \end{multline*}
where $\int$ denotes $\frac{1}{T}\int_T^{2T}$ for short. 

Note that, by \eqref{excep set P_j} and \eqref{extra term bound} we have 
\begin{align*}
\exp\Big(\frac{C_M\log T}{\log Z_\mf{j}}\Big)
\cdot
\frac{1}{T}\int_T^{2T} \bigg(\frac{|{P}_{\mf{j}+1,l}(t)|}{\ell_{\mf{j}+1}}\bigg)^{2r_\mf{j}} dt
\ll &
 \exp\Big(-\frac{C\log T}{e^{\mf{j}+1}(\log_2 T)^{1+\epsilon}}\Big)
 \\
 \ll &
 \exp\big(-(\log T)^{1/2+o(1)}\big).
\end{align*}
Since the number of terms in the sum over $\mf{j},l$ along with the $(\log T)^{O(1)}$ term can be absorbed into this exponential, we arrive at 
\begin{equation}
\begin{split}
\label{M int bound}
&
\frac{1}{T}\int_T^{2T}|M(\tfrac12+it)|^2|R(t)|^2dt 
\\
\ll &
\exp\bigg(\sqrt{\frac{\log T}{\log_2 T\log_3 T}}\bigg)
\max_{\mf{j}\leqs J}
  \frac{1}{T}\int_T^{2T} |\mc{N}_{0,\mf{j}}(t)|^2|R(t)|^2dt 
\prod_{\mf{i}=1}^\mf{j} \frac{1}{T}\int_T^{2T}\big|\mathcal{N}_{\mf{i},\mf{j}}(t)\big|^2dt.
\end{split}
\end{equation}

\subsection{Computing the mean values}

It remains to compute 
\[
 \frac{1}{T}\int_T^{2T} |\mc{N}_{0,\mf{j}}(t)|^2|R(t)|^2dt 
\cdot
\prod_{\mf{i}=1}^\mf{j} \frac{1}{T}\int_T^{2T}\big|\mathcal{N}_{\mf{i},\mf{j}}(t)\big|^2dt.
\]
Applying the mean value theorem for Dirichlet polynomials we have 
\begin{align*}
\frac{1}{T}\int_T^{2T}|\mc{N}_{0,\mf{j}}(t)|^2|R(t)|^2dt 
=
(1+O(T^{-\epsilon}))
\sum_{\substack{m_1n_2=m_2n_2\\\Omega(m_j)\leqs \ell_0\\p|m_j\implies p\leqs Z_0\\n_j\leqs X}}
\frac{\mf{c}_\mf{j}(m_1)\ol{\mf{c}_\mf{j}(m_2)}r(n_1)\ol{r(n_2)}}{(m_1m_2)^{1/2}}
\end{align*}
and 
\begin{align*}
\prod_{\mf{i}=1}^\mf{j}\frac{1}{T}\int_T^{2T}|\mc{N}_{\mf{i},\mf{j}}(t)|^2dt 
=
(1+O(T^{-\epsilon}))\prod_{\mf{i}=1}^\mf{j}
\sum_{\substack{\Omega(m)\leqs \ell_\mf{i}\\p|m\implies Z_{\mf{i}-1}<p\leqs Z_\mf{i}}}
\frac{|\mf{c}_\mf{j}(m)|^2}{m}
\end{align*}
with
\[
\mf{c}_\mf{j}(m)=c(m)\mf{g}(m)W_\mf{j}(m)
\]
where we recall the definition of these coefficients from \eqref{coeffs 1}, \eqref{coeffs 2}.

Now,
\begin{equation}\label{c sum bound}
\prod_{\mf{i}=1}^\mf{j}
\sum_{\substack{p|n\implies Z_{\mf{i}-1}<p\leqs Z_\mf{i}}}
\frac{|\mf{c}_\mf{j}(m)|^2}{m}
\leqs \exp\bigg(\sum_{Z_0<p\leqs Z_\mf{j}}\frac{|b(p)|^2}{p}\bigg)
\ll
\bigg(\frac{\log Z_J}{\log Z_0}\bigg)^{m-1}
\end{equation}
by \eqref{b prime sum}.
It thus suffices to show
\begin{equation}\label{J bound}
\begin{split}
\sum_{\substack{m_1n_2=m_2n_2\\\Omega(m_j)\leqs \ell_0\\p|m_j\implies p\leqs Z_0\\n_j\leqs X}}
\frac{\mf{c}_\mf{j}(m_1)\ol{\mf{c}_\mf{j}(m_2)}r(n_1)\ol{r(n_2)}}{(m_1m_2)^{1/2}} 
\ll &
\prod_p\bigg(1+|r(p)|^2+2(1+o(1))\Re\frac{r(p)\ol{b(p)}}{p^{1/2}}\bigg).
\end{split}
\end{equation}
Assuming this for the moment, plugging \eqref{c sum bound} and \eqref{J bound}  into \eqref{M int bound} gives 
\begin{multline*}
\frac{1}{T}\int_T^{2T}|M(\tfrac12+it)|^2|R(t)|^2dt
\\
\ll 
\exp\bigg(\sqrt{\frac{\log T}{\log_2 T\log_3 T}}\bigg) 
\prod_p\bigg(1+|r(p)|^2+2(1+o(1))\Re\frac{r(p)\ol{b(p)}}{p^{1/2}}\bigg).
\end{multline*}
and Proposition \ref{gen upper bound prop} follows.

To prove \eqref{J bound} we apply Rankin's trick to find that the sum on the left there is for any $\alpha>0$,
\begin{align}\label{c mega sum bound}
&\sum_{\substack{m_1n_2=m_2n_2\\p|n\implies p\leqs Z_0}}
\frac{\mf{c}_\mf{j}(m_1)\ol{\mf{c}_\mf{j}(m_2)}r(n_1)\ol{r(n_2)}}{(m_1m_2)^{1/2}} 
\\
&+
O\bigg(
e^{-\ell_0}\sum_{\substack{m_1n_2=m_2n_2\\p|m_j\implies p\leqs Z_0}}
\frac{|\mf{c}_\mf{j}(m_1){\mf{c}_\mf{j}(m_2)}r(n_1){r(n_2)}|e^{\Omega(m_1)}}{(m_1m_2)^{1/2}} 
\bigg)
\\
&+
O\bigg(
X^{-\alpha}\sum_{\substack{m_1n_2=m_2n_2\\p|m_j\implies p\leqs Z_0}}
\frac{|\mf{c}_\mf{j}(m_1){\mf{c}_\mf{j}(m_2)}r(n_1){r(n_2)}|n_1^\alpha}{(m_1m_2)^{1/2}} 
\bigg)
\end{align}
by symmetry. 

The main term here is 
\begin{align*}
&
\prod_{p\leqs Z_0}\sum_{m_1+n_1=m_2+n_2}\frac{b(p)^{m_1}\ol{b(p)^{m_2}}w_\mf{j}(p)^{m_1}w_\mf{j}(p)^{m_2}r(p^{n_1})\ol{r(p^{n_2})}}{m_1!m_2!p^{(m_1+m_2)/2}}
\\
= &
\prod_{p\leqs Z_0}\bigg((1+|r(p)|^2)\sum_{m\geqs 0}\frac{|b(p)|^{2m}w_\mf{j}(p)^{2m}}{m!^2p^m} 
+
2\Re\frac{r(p)\ol{b(p)}}{p^{1/2}}\sum_{m\geqs 0}\frac{|b(p)|^{2m}w_\mf{j}(p)^{2m}}{m!(m+1)!p^m} \bigg)
\\
= &
\mc{K}(X)\prod_{p}\bigg(1+|r(p)|^2+2\Re\frac{r(p)\ol{b(p)}}{p^{1/2}}B(p)\bigg)
\end{align*}
where 
\[
\mc{K}(X)=\prod_{p\leqs Z_0}\sum_ {m\geqs 0}\frac{|b(p)|^{2m}w_\mf{j}(p)^{2m}}{m!^2p^m} \ll\exp\bigg(\sum_{p\leqs Z_0}\frac{|b(p)|^2}{p}\bigg)\ll(\log Z_0)^{m-1},
\]
by \eqref{b prime sum}  and 
\[
B(p)=\frac{\sum_{m\geqs 0}{|b(p)|^{2m}w_\mf{j}(p)^{2m}}/{m!^2p^m} }{\sum_{m\geqs 0}{|b(p)|^{2m}w_\mf{j}(p)^{2m}}/{m!(m+1)!p^m}}=1+O\Big(\frac{1}{p^{2/(\max d_i^2+1)}}\Big)
\]
by \eqref{a bound}.
Since $B(p)=1+o(1)$ for $p$ in the support of $r$ and $(\log Z_0)^{m-1}$ can be absorbed into this $o(1)$ term of the exponential, it remains to show that the  error terms of \eqref{c mega sum bound} are of a lower order than this.

With a similar calculation the first error term there is 
\begin{align*}
\ll
e^{-\ell_0}(\log Z_0)^{e(m-1)}\prod_{p}\bigg(1+|r(p)|^2+2e(1+o(1))\frac{|r(p)b(p)|}{p^{1/2}}\bigg).
\end{align*}
Since $|r(p)|^2, r(p)b(p)=o(1)$ in the support of $r(\cdot)$ as in \eqref{smallr}, the ratio of this to the main term is then 
\[
\ll e^{-\ell_0}(\log Z_0)^{e(m-1)}\exp\bigg(4e\sum_{p\leqs Z_0}\frac{|r(p)b(p)|}{p^{1/2}}\bigg)=o(1)
\]
on recalling that $\ell_0=\log T/100(\log\mc{L})^2\asymp \log T/(\log_2 T)^2$ and noting that the sum in the exponential is $\ll \sqrt{\log T/\log_2 T}$.

The second error term is 
\begin{align*}
\ll
X^{-\alpha}(\log Z_0)^{m-1}\prod_{p}\bigg(1+|r(p)|^2p^\alpha+2(1+o(1))\frac{|r(p)b(p)|}{p^{1/2-\alpha}}\bigg).
\end{align*}
The ratio of this to the main term is
\[
\ll
\exp\bigg(-\alpha\log X+\sum_{\mc{L}^2<p\leqs \exp((\log \mc{L})^2)}|a(p)|^2\frac{\mc{L}^2}{p\log^2 p}(p^{\alpha}-1)+O\Big(\sqrt{\frac{\log T}{\log_2 T}}\Big)\bigg)
\]
for $\alpha=1/(\log \mc{L})^3$. The usual computations, as in \eqref{Rankin error}, show this is $o(1)$.


\section{Simultaneous extreme values of twists of $GL(2)$ cusp forms:\\ Proof of Theorem \ref{mod q twist thm}}

Let $f,g$ be a fixed primitive (holomorphic or Maa\ss) cusp forms with respect to $\Gamma_0(r)$ and $\Gamma_0(r')$ with trivial central character. 
	As before, if we can find $R(\chi)$ and $V$ such that 
	\begin{align}\label{chi q ineq}
		&\frac{1}{\phi^*(q)}\sideset{}{^*}\sum_{\chi\bmod q} |L(1/2,f\otimes \chi)\overline{ L(1/2,g\otimes \chi)}R(\chi)|^2
		\\
		&\geqs 
		 \frac{V}{\phi^*(q)}\sideset{}{^*}\sum_{\chi\bmod q} \Big(|L(1/2,f\otimes \chi)R(\chi)|^2+|L(1/2,g\otimes \bar\chi)R(\chi)|^2\Big)
	\end{align}
	with $*$ meaning the sum is over primitive characters modulo $q$ and $\phi^*(q)=q-2$, 
	then we must have 
	\begin{align*}
		\max_{\chi\bmod q} \min_{f,g}\Big(|L(1/2,f\otimes \chi)|, L(1/2,g\otimes\chi)|\Big)\geqs \sqrt{V}.
	\end{align*}
	To estimate these mean values we follow \cite{BFKMMS} quite closely and so retain some of their methods, notation and set-up for ease of comparison, although it may differ from our previous sections slightly. 
	
	 We consider the sum on the left hand side of \eqref{chi q ineq} first. 
	From Cauchy's inequality we have
	\begin{align*}
		&\frac{1}{\phi^*(q)}\sideset{}{^*}\sum_{\chi \bmod q}|L(1/2,f\otimes \chi)\overline{L(1/2,g\otimes\chi)}R(\chi)|^2\\
		&\geqs \Big(\frac{1}{\phi^*(q)}\sideset{}{^*}\sum_{\chi \bmod q} |R(\chi)|^2\Big)^{-1}\Big( \frac{1}{\phi^*(q)}
		\sideset{}{^*}\sum_{\chi \bmod q} L(1/2,f\otimes \chi)\overline{L(1/2,g\otimes \chi)}|R(\chi)|^2\Big)^2
	\end{align*}
	Thus it make sense to choose $R(\chi)$ such that $|L(1/2,f\otimes \chi)L(1/2,g\otimes\chi)|$ is large and we follow the choice of $R(\chi)$ in \cite[section 7.5.1]{BFKMMS}.
	
	Let $\lambda_f^*, \lambda_g^*$ be multiplicative functions supported on squarefree positive integers defined by $\lambda_f^*(p)=(1-1/p)^{-1}(\lambda_f(p)-\lambda_g(p)/p)$ and $\lambda_g^*(p)=(1-1/p)^{-1}(\lambda_g(p)-\lambda_f(p)/p)$. For 
	$u\geqs 1$ some parameter depending only on $f$ and $g$, let \begin{align}\label{Gdef}
		\mathcal G:=\{n\geqs 1: (n, urr')=1, \lambda_f^*(n)\lambda_g^*(n)\not=0, \operatorname{sgn}(\lambda_f^*(n))=\operatorname{sgn}(\lambda_g^*(n))\}.
	\end{align}
	Define 
	\begin{align*}
		\varpi(p)=\begin{cases}
			\lambda_f^*(p)\lambda_g^*(p)(\lambda_f^*(p)+\lambda_g^*(p)), & p\in \mathcal G,\\
			0, & p\not \in \mathcal G. 
		\end{cases}
	\end{align*}
	Let 
	\begin{align*}
		\omega(n)=|\varpi(n)|^2,\  \omega_1'(n)=\varpi(n)\lambda_f^*(n),\  \omega_2'(n)=\varpi(n)\lambda_g^*(n),\  \omega'(p)=\omega_1'(p)+\omega_2'(p)
	\end{align*}
	and     	$$R(\chi)=\sum_{n\leqs N}r(n)\varpi(n)\chi(n)$$
where
	\begin{equation}\label{r 2}
r(p)
=
\begin{cases}
\frac{\mc{L}}{p^{1/2}\log p}{} \,\,\, &\text{ for } \mc{L}^2\leqs p \leqs \exp((\log \mc{L})^2)
\\
0  & \text{ otherwise }
\end{cases}
\end{equation}
and 
\[
\mc{L}=\sqrt{a_{\omega}^{-1}\log N\log\log N}
\] 
for some constant $a_\omega$ as in \cite[eq. (7.55)]{BFKMMS}
	
	Fixing an arbitrary $\delta>0$ and $N\leqs q^{1/360-\delta}$ we have that using \cite[Lemma 7.19, Lemma 7.10]{BFKMMS}
	\begin{align*}
		\frac{1}{\phi^*(q)}\sideset{}{^*}\sum_{\chi \bmod q}|R(\chi)|^2\sim \prod_{p}\Big(1+r(p)^2\omega(p)\Big)
	\end{align*}
	and using \cite[Lemma 7.19, Lemma 7.12, Lemma 7.14]{BFKMMS} there exists a squarefree integer $u\geq 1$ coprime to $rr'$ such that 
	\begin{align*}
		&\frac{1}{\phi^*(q)}\sideset{}{^*}\sum_{\chi \bmod q}|R(\chi)|^2 L(1/2,f\otimes \chi)L(1/2,g\otimes \bar{\chi})\chi(u)\\
		=&L^*(1,f\otimes g)(\nu+o(1))\prod_{p}\Big(1+r(p)^2 \omega(p)+\frac{r(p)\omega'(p)}{\sqrt{p}}\Big)
		+O\Big(q^{-\delta}\prod_{p}\big(1+r(p)^2\omega(p)\big)\Big),
	\end{align*}
	where $L^*(s,f\otimes g)$ is as in \cite[eq. (2.7)]{BFKMMS},   $\nu\not=0$ is a constant depending on $f,g$ only. As in \cite{BFKMMS}, the $\chi(u)$ is introduced to break the symmetry of $\chi$ and $\ol{\chi}$. Since $|\chi(u)|\leqs1$ and $L^*(f\otimes g, 1)\not=0$ (\cite[Lemma 2.6]{BFKMMS}) we see that 
	\begin{align}\label{fgchilowerbound}
		&\frac{1}{\phi^*(q)}\sideset{}{^*}\sum_{\chi \bmod q}|L(1/2,f\otimes \chi)\overline{L(1/2,g\otimes \chi)}R(\chi)|^2 \\&\gg_{f,g} \prod_{p}\Big(1+r(p)^2 \omega(p)+\frac{r(p)w'(p)}{\sqrt{p}}\Big)^2 \Big(1+r(p)^2 \omega(p)\Big)^{-1}.
	\end{align}
	
	We now turn to getting an upper bound for the mean squares on the right of \eqref{chi q ineq}. Similarly to the proof of \cite[Lemma 7.9]{BFKMMS}, we have for $(\ell, \ell')=(\ell\ell', qrr')=1$, $\ell,\ell'\leqs L$,
	\begin{multline*}
	\frac{1}{\phi^*(q)}\sideset{}{^*}\sum_{\chi \bmod q} |L(1/2,f\otimes \chi)|^2\chi(\ell) \chi(\ell')=\frac{1}{2}\operatorname{MT}^{+}(f,f;\ell, \ell')+\frac{1}{2}\operatorname{MT}^{-}(f,f;\ell, \ell')
	+O(L^{3/2}q^{-1/144+\epsilon}),
	\end{multline*}
	where
	\begin{align*}
	\operatorname{MT}^{\pm}(f,f;\ell, \ell')&=\frac{1}{2\pi i}\int_{(2)}\frac{ L_\infty^\pm (f, \pm, \frac{1}{2}+ u)}{L_\infty^2(f, \pm,\frac{1}{2})}\frac{D(1+2u;\ell,\ell')}{(\ell \ell')^{1/2+u}}G(u)(q^2|rr'|)^u\frac{du}{u},\\
	D(s; \ell, \ell')&=\sum_{n}\frac{\lambda_f(\ell n)\lambda(\ell'n)}{n^s}=L^*(f\otimes f, s)\prod_{p\mid\ell \ell'}(1-p^{-2s})^{-1}\prod_{p\mid \ell\ell'}(\lambda_f(p)-\frac{\lambda_f(p)}{p^s}).
	\end{align*}
	Here $G(u)=\cos(\frac{\pi u}{4A})^{-16 A}$ for some $A\geq 2$ and $L_\infty(f, \pm, s)=L_\infty(f\otimes\chi, s)$ for $\chi(-1)=\pm1$ (see \cite[Lemma 2.1]{BFKMMS} for definitions of $L_\infty (f\otimes \chi, s)$).
	Shifting the contour of integration to $\Re(u)=-\frac{1}{4}+\epsilon$, we encounter a double pole at $u=0$ so that \begin{align*}
	\operatorname{MT}^\pm= \frac{\lambda_f^*(\ell\ell')}{\sqrt{\ell\ell'}}\Big(\frac{L^*(\operatorname{Sym}^2f, 1)}{\prod_{p\mid r}(1+p^{-1})\zeta(2)}+2\log (|r|q)-\log(\ell\ell')+C_f+\sum_{p\mid \ell \ell'}\frac{2\log p}{p+1}\Big)
	\end{align*}
	using \cite[eq. (2.9)]{BFKMMS} for $\operatorname{Res}_{s=1}L^*(f\otimes f, 1)$ and $C_f=\frac{d}{du}\frac{L_\infty^\pm(f, \pm, \frac{1}{2}+u)}{L^2_\infty(f, \pm , 1/2)}\Big\vert_{u=0}$. 
	
	Therefore, we have that 
	\begin{multline*}
		\frac{1}{\phi^*(q)}\sideset{}{^*}\sum_{\chi \bmod q}|L(1/2,f\otimes \chi)R(\chi)|^2
		=\sum_{d}|r(d)|^2 \omega(d)\sum_{\substack{\ell, \ell'\leq N/d\\ (\ell\ell',d)=1}}\frac{r(\ell\ell')\varpi(\ell \ell')}{\sqrt{\ell \ell'}}\lambda_f^*(\ell\ell')
		\\
		\times \Big(A_f+2\log q-\log (\ell \ell')+\sum_{p\mid \ell \ell'} \frac{2\log p}{p+1}+O(N^{3/2}q^{-1/144})\Big)
	\end{multline*}
	where $A_f=\frac{L^*(\operatorname{Sym}^2f, 1)}{\prod_{p\mid r}(1+p^{-1})\zeta(2)}+C_f+2\log |r|$.
	It follows that 
	\begin{align}
		&\frac{1}{\phi^*(q)}\sideset{}{^*}\sum_{\chi \bmod q}|L(1/2,f\otimes \chi)R(\chi)|^2 +O(N^{7/2}q^{-144})
		\\&
		=\sum_{d}|r(d)|^2 \omega(d)\sum_{\substack{\ell, \ell'\leq N/d\\ (\ell\ell',d)=1}}\frac{r(\ell \ell')\varpi(\ell \ell')}{\sqrt{\ell \ell'}}\Big(A_f\lambda_f^*(\ell \ell')
		+\lambda_f^*(\ell\ell')
		\big(2\log q-\log (\ell\ell')
		+2\sum_{p\mid \ell \ell'}\frac{\log p}{p+1}\big)\Big)		
	\end{align}
	Note from the support of $\varpi$ in \eqref{Gdef}, we have $0\leqs \varpi(p)\lambda_f^*(p)=\omega_1'(p)$ and thus 
	\begin{align*}
	\sum_{\substack{\ell, \ell'\leq N/d\\ (\ell\ell',d)=1}}\frac{r(\ell\ell')\varpi(\ell \ell')\lambda_f^*(\ell\ell')}{\sqrt{\ell \ell'}}\leq \prod_{\substack{\ell, \ell'\\ (\ell\ell',d)=1}}\frac{r(\ell\ell')\varpi(\ell\ell')\lambda_f^*(\ell\ell')}{\sqrt{\ell\ell'}}=\prod_{\substack{\ell, \ell'\\ (\ell\ell',d)=1}}\frac{r(\ell\ell')\omega_1'(\ell\ell')}{\sqrt{\ell\ell'}}.
	\end{align*} 
	Therefore
	\begin{align*}
	&\sum_{d}|r(d)|^2 \omega(d)\sum_{\substack{\ell, \ell'\leq N/d\\ (\ell \ell',d)=1}}\frac{r(\ell\ell')\varpi(\ell\ell')}{\sqrt{\ell\ell'}}A_f\lambda_f^*(\ell\ell')
	\ll_f \prod_{p}\Big(1+r(p)^2 \omega(p)+\frac{2r(p)\omega_1'(p)}{\sqrt{p}}\Big).
	\end{align*}
	With the current choice of $\varpi, N$, we also have that  $2\log q-\log(\ell\ell')+\sum_{p\mid \ell \ell'}\frac{2\log p}{p+1}\geqs 0$, which together with $\varpi(\ell)\lambda_f^*(\ell)\geq 0$ gives 
	\begin{align*}
		&\sum_{d}|r(d)|^2 \omega(d)\sum_{\substack{\ell, \ell'\leq N/d\\ (\ell\ell',d)=1}}\frac{r(\ell\ell')\varpi(\ell\ell')}{\sqrt{\ell\ell'}}\Big(2\log q-\log(\ell\ell')+\sum_{p\mid \ell \ell'}\frac{2\log p}{p+1}\Big)\lambda_f^*(\ell\ell')\\
		&\ll\log q\sum_{d}|r(d)|^2 \omega(d)\sum_{\substack{\ell, \ell'\\ (\ell\ell',d)=1}}\frac{r(\ell\ell')\varpi(\ell\ell')}{\sqrt{\ell\ell'}}\lambda_f^*(\ell\ell')\\
		&\ll\ \log q\prod_{p}\Big(1+r(p)^2\omega(p)+\frac{2r(p)\omega_1'(p)}{\sqrt{p}}\Big).
	\end{align*}
	Therefore, we have 
	\begin{align}\label{fgchiupperbound}
		&\frac{1}{\phi^*(q)}\sideset{}{^*}\sum_{\chi \bmod q}\Big(|L(1/2,f\otimes \chi)R(\chi)|^2+ |L(1/2,g\otimes \chi)R(\chi)|^2\Big)
		\\
		\ll & \max_{i=1,2}\ \log q\prod_{p}\Big(1+r(p)^2\omega(p)+\frac{2r(p)\omega_i'(p)}{\sqrt{p}}\Big).
	\end{align}
	Combining \eqref{fgchilowerbound} and \eqref{fgchiupperbound}, 
 we can choose $$V=\min_{i=1,2}\frac{1}{\log q}\prod_{p}\Big(1+r(p)^2 \omega(p) +\frac{r(p)\omega'(p)}{\sqrt{p}}\Big)^2(1+r(p)^2 \omega(p))^{-1}\Big(1+r(p)^2 \omega(p)+\frac{2r(p)\omega_i'(p)}{\sqrt{p}}\Big)^{-1}.$$
	We have from \cite[Proof of Lemma 7.5]{BFKMMS} that
	\begin{align*}
	&\log\prod_{p}\Big(1+r(p)^2 \omega(p) +\frac{r(p)\omega'(p)}{\sqrt{p}}\Big)^2(1+r(p)^2 \omega(p))^{-1}\Big(1+r(p)^2 \omega(p)+\frac{2r(p)\omega_i'(p)}{\sqrt{p}}\Big)^{-1}\\
	&=\log\prod_{p}\Big(1 +\frac{r(p)\omega'(p)}{\sqrt{p}(1+r(p)^2\omega(p))}\Big)^2\Big(1+\frac{2r(p)\omega_i'(p)}{\sqrt{p}(1+r(p)^2\omega(p))}\Big)^{-1}\\
	&=\mc{L}\sum_{L^2 \leq p\leq \exp(\log^2 L)}\frac{2\omega'(p)-2\omega_i'(p)}{p\log p}+O_{\omega, \omega', \delta}\Big(\frac{L}{(\log L)^{1+\delta}}\Big)
	\end{align*}
Since $\lambda_f^*(p)\lambda_g^*(p)(\lambda_f^*(p)+\lambda_g^*(p))\lambda_f^*(p),  \lambda_f^*(p)\lambda_g^*(p)(\lambda_f^*(p)+\lambda_g^*(p))\lambda_g^*(p)\leqs 0$ when $p\not \in \mathcal G$, we have 
\begin{align*}
\sum_{L^2\leq p\leq \exp (\log^2 L)} \frac{2\omega'(p)-2\omega_1'(p)}{p\log p}&\geqs\sum_{\substack{\mc{L}^2 \leqs p\leqs \exp (\log^2 \mc{L})\\ p\in \mathcal G}} \frac{2\lambda_f^*(p)\lambda_g^*(p)(\lambda_f^*(p)+\lambda_g^*(p))\lambda_g^*(p)}{p\log p}\\
& \geqs\sum_{\substack{\mc{L}^2 \leq p\leq \exp (\log^2 \mc{L})}} \frac{2\lambda_f^*(p)\lambda_g^*(p)(\lambda_f^*(p)+\lambda_g^*(p))\lambda_g^*(p)}{p\log p}\\
& = \frac{c}{\log L}+O(\frac{1}{(\log L)^2}).
\end{align*}
for some positive $c=2n_{2,2}+2n_{1,3}$ using the notation for $n_{i,j}$ in \cite[Corollary 2.17]{BFKMMS}. Thus we see that for every prime $q$ sufficiently large depending on $f, g$, there exists a non-trivial character $\chi \bmod q$ such that 
\begin{align*}
\min_{f,g }(|L(1/2,f\otimes \chi)|, |L(1/2,g\otimes\chi)|) \geqs \sqrt{V}\geqs \exp \bigg(c_{f,g}\sqrt{\frac{\log q}{\log \log q}}\bigg)
\end{align*}
for some positive $c_{f,g}$. 
With $N=q^{1/360-\delta}$, we can take the constant $$c_{f,g}=\frac{1}{2}(\frac{1}{6\sqrt{10}}\sqrt{1-360\delta}+o(1) )\frac{n_{2,2}+\min(n_{1,3}, n_{3,1})}{(n_{4,2}+2n_{3,3}+n_{2,4})^{1/2}}.$$
In a generic situation, as in \cite[Remark 7.20]{BFKMMS}, i.e. where neither $f$ nor $g$ are of polyhedral type (in particular $\operatorname{Sym^k}f, \operatorname{Sym}^kg$ are cuspidal for all $k\leq 4$) and if $\operatorname{Sym}^k \pi_f\not \cong \operatorname{Sym}^k \pi_g$ for $k\leqs 4$, we see that 
\begin{align}
c_{f,g}=\frac{1}{12\sqrt{10}}+o(1).
\end{align} 

\section{Simultaneous small values of quadratic twists: Proof of Theorem \ref{quad twist thm}}

Let $f,g$ be holomorphic cusp forms of weight $\kappa\equiv 0 \bmod 4$ for $SL_2(\mathbb{Z})$ and let $\chi_{d}(n)=\big(\frac{d}{n}\big)$ be the Kronecker symbol. 
Due to the non-negativity of $L(1/2,f\otimes \chi_{d})$, small values of $L(1/2,f\otimes \chi_{d})+L(1/2,g\otimes \chi_{d})$ implies simultaneous small values of $L(1/2, f\otimes \chi_{d})$ and $L(1/2, g\otimes \chi_{d})$.  
Let $\Phi(x):(0, \infty)\rightarrow \mathbb{C}$ be a smooth, compactly supported function. From \cite[Theorem 1.4]{Shen}, we have for square-free $\ell$ 
\begin{align*}
\sideset{}{^*}\sum_{(d,2)=1}\chi_{8d}(\ell)L(1/2,f\otimes \chi_{8d})\Phi(\frac{d}{X})=\frac{8X \tilde{\Phi}(1)}{\pi^2 \sqrt{\ell}} L(1,\operatorname{Sym}^2f)Z(1/2, \ell)+O(\ell^{1/2+\epsilon}X^{1/2+\epsilon})
\end{align*}
%
%
where the $*$ now denotes a sum over squarefree integers and 
\begin{align*}
L(1,\operatorname{Sym}^2f)Z(1/2,\ell)
=&
\prod_{p\mid \ell}\frac{p^{3/2}}{2(p+1)}\Big((1-\frac{\lambda_f(p)}{\sqrt{p}}+\frac{1}{p})^{-1}-(1+\frac{\lambda_f(p)}{\sqrt{p}}+\frac{1}{p})^{-1}\Big)
\\& \times \prod_{p\nmid 2\ell}\Big(1+\frac{p}{2(p+1)}\big((1-\frac{\lambda_f(p)}{\sqrt{p}}+\frac{1}{p})^{-1}+(1+\frac{\lambda_f(p)}{\sqrt{p}}+\frac{1}{p})^{-1}-2\big)\Big).
\end{align*}
We write $L(1,\operatorname{Sym}^2f)Z(1/2,\ell)=:
L(1,\operatorname{Sym}^2 f)Z(1/2, 1) \prod_{p\mid \ell}h_f(p)$ so that 
\begin{align*}
h_f(p)&=\frac{p^{3/2}}{2(p+1)}\Big((1-\frac{\lambda_f(p)}{\sqrt{p}}+\frac{1}{p})^{-1}-(1+\frac{\lambda_f(p)}{\sqrt{p}}+\frac{1}{p})^{-1}\Big)\\& \times\Big(1+\frac{p}{2(p+1)}\big((1-\frac{\lambda_f(p)}{\sqrt{p}}+\frac{1}{p})^{-1}+(1+\frac{\lambda_f(p)}{\sqrt{p}}+\frac{1}{p})^{-1}-2\big)\Big)^{-1}
\\&=\lambda_f(p)+ O\big(\frac{|\lambda_f(p)|}{p}\big).
\end{align*}
Thus for $R(\chi_{8d})=\sum_{n\leq N}\mu(n)r(n)\varpi(n)\chi_{8d}(n)$ with $r, \varpi$ real multiplicative functions supported on squarefree integers
\begin{align}\label{dn1n2}
&\sideset{}{^*}\sum_{(d, 2)=1}L( 1/2,f\otimes \chi_{8d})|R(\chi_{8d})|^2\Phi(\frac{d}{X})
\\
&=\frac{8X\tilde{\Phi}(1)}{\pi^2}L(1, \operatorname{Sym}^2f)Z(1/2,1)\sum_{d\leq N}r(d)^2\omega(d)\sum_{\substack{n_1, n_2\leq N/d\\ (n_1n_2,d)=1}}\frac{\mu(n_1n_2)r(n_1n_2)\omega_1'(n_1n_2)}{\sqrt{n_1n_2}}
\\
&\quad +O( N^{5/2+\epsilon}X^{1/2+\epsilon})\\
\end{align}
where $\omega(n)=|\varpi(n)|^2$ and $\omega_1'(n)=\varpi(n)h_f(n)$. A similar expression holds when $f$ is replaced by $g$. Now it remains to find $r(n)$ and $\varpi(n)$. 

As usual, let $r(n)$ be multiplicative supported on squarefrees and satisfying \eqref{r 2} with $\mc{L}=\sqrt{a_{\omega}\log N\log\log N}$ where $a_\omega$ is defined as in \cite[eq (7.2)]{BFKMMS}.
Let $
\tilde{\mathcal G}:=\{n\geq 1: h_f(n)h_g(n)\not=0, \operatorname{sgn}(h_f(n))=\operatorname{sgn}(h_g(n))\}
$
and define $\varpi$ as
\begin{align}\label{varpidef3}
\varpi(p)=\begin{cases}
h_f(p)h_g(p)(h_f(p)+h_g(p)), & p\in \tilde{\mathcal G},\\
0, & p\not \in \tilde{\mathcal G}. 
\end{cases}
\end{align}
  Then
similarly as before, we can evaluate the $d, n_1, n_2$-sum in \eqref{dn1n2} as 
\begin{align*}
(1+o(1))\prod_{p}\Big(1+r(p)^2\omega(p)-\frac{2r(p)\omega_1'(p)}{\sqrt{p}}\Big).
\end{align*}
With $N=X^{1/5-\delta}$ we see that \eqref{dn1n2} becomes
\begin{align}
&\frac{8X\tilde{\Phi}(1)}{\pi^2}L(\operatorname{Sym}^2f, 1/2)Z(1/2,1)(1+o(1))\prod_{p}\Big(1+r(p)^2 \omega(p)-\frac{2r(p)\omega_1'(p)}{\sqrt{p}}\Big) +O(X^{1-5\delta/2+\epsilon}).
\end{align}
By standard computations (e.g. see \cite{Sound res}), we have 
\[
\sideset{}{^*}\sum_{(d,2)=1}|R(\chi_{8d})|^2\Phi(\tfrac{d}{X})\sim cX\prod_p\big(1+r(p)^2\omega(p))
\]
for some positive constant $c$.

Thus, since $h_f(p)=\lambda_f(p)+O(p^{-1+\theta})$ with $\theta=7/64$, we have from \cite[Corollary 2.17, Lemma 7.19, proof of Lemma 7.5]{BFKMMS} that our ratio of mean values is 
\begin{align*}
\ll
\prod_{p}\Big(1-\frac{2r(p)\omega_1'(p)}{\sqrt{p}(1+r(p)^2\omega(p))}\Big)
\leqs
 \exp\Big(- (n_{3,1}+n_{2,2}+o(1))\frac{\mc{L}}{2\log \mc{L}}\Big)
\end{align*}
with $n_{i,j}$ defined as in \cite[Corollary 2.17]{BFKMMS}. 
Therefore, we see that there exists $d$ such that 
\begin{align*}
\max(L(1/2,f\otimes \chi_{8d}), L(1/2,g\otimes \chi_{8d}))\leqs \exp\Big(-\tilde{c}_{f,g}(\sqrt{\frac{1}{5}-\delta})\sqrt{\frac{\log X}{\log\log X}}\Big).
\end{align*}
where positive $\tilde{c}_{f,g}=\frac{\min\big(n_{3,1}, n_{1,3}\big)+n_{2,2}}{\sqrt{a_\omega}}+o(1)>0$. In a generic situation, where neither $f$ nor $g$ are of polyhedral type (in particular $\operatorname{Sym^k}f, \operatorname{Sym}^kg$ are cuspidal for all $k\leq 4$) and if $\operatorname{Sym}^k \pi_f\not \cong \operatorname{Sym}^k \pi_g$ for $k\leq 4$, then 
$\tilde{c}_{f,g}=1+o(1)$ using \cite[eq. (7.55)]{BFKMMS} for $a_\omega$.


\end{document}